\documentclass{amsart}
\usepackage{graphicx}
\usepackage{amsmath}
\usepackage{amscd}
\usepackage{mathtools}
\usepackage{enumerate}
\usepackage{tikz-cd}
\usepackage[all,cmtip]{xy}
\usepackage{tikz}
\usepackage{amssymb}
\usepackage{ascmac} 
\usepackage{color}
\usepackage{breqn}
\newtheorem{dfn}{Definition}[section]
\newtheorem{thm}[dfn]{Theorem}
\newtheorem{pro}[dfn]{Proposition}
\newtheorem{lem}[dfn]{Lemma}
\newtheorem{cro}[dfn]{Corollary}
\theoremstyle{definition}

\newcommand{\ff}{F(X_0,X_1,Y_0,Y_1)}
\newcommand{\fa}{F_{Q_1}}
\newcommand{\fb}{F_{Q_2}}
\newcommand{\fc}{F_{Q_3}}
\newcommand{\fd}{F_{Q_4}}
\newcommand{\co}{\colon}
\newcommand{\ra}{\rightarrow}
\newcommand{\pp}{\mathbb P^1\times \mathbb P^1}
\newcommand{\s}{s_{(1,2)}}
\newcommand{\ta}{the\ above\ }
\newcommand{\mI}{\mathbb I}
\title[Large orders of automorphisms]{Large orders of automorphisms of smooth curves in $\mathbb P^1\times \mathbb P^1$}
\author{Taro Hayashi}
\author{Keika Shimahara}
\address{
	(Taro Hayashi)
	Department of Mathematical Sciences,
	Ritsumeikan University,
	1$-$1$-$1 Nojihigashi, Kusatsu, Shiga, 525$-$8577, Japan}
\email{haya4taro@gmail.com}
\address{
	(Keika Shimahara)
Graduate School of Mathematical Sciences,
	Ritsumeikan University,
	1$-$1$-$1 Nojihigashi, Kusatsu, Shiga, 525$-$8577, Japan}
\email{ra0134hp@ed.ritsumei.ac.jp}
\date{\today}
\begin{document}
\maketitle
\begin{abstract}
For $a,b\geq 3$, we calculate the orders of automorphisms of smooth curves with bidegree $(a,b)$ in the product $\pp$ of the projective line $\mathbb P^1$.
We identify smooth curves in $\pp$ which have automorphisms with the largest orders.
In addition, we study the relationship between symmetry and geometric structure of curves.
We provide a sufficient condition for the quotient space by an automorphism to be $\mathbb P^1$.
\end{abstract}
\section{Introduction}
In this paper, we work over ${\mathbb C}$. 
For a variety $X$, let ${\rm Aut}(X)$ be the automorphism group of $X$.
For an automorphism $g\in{\rm Aut}(X)$, 
we write the order of $g$ as ${\rm ord}(g)$ and the fixed points set of $g$ as ${\rm Fix}(g)$.
Let $\langle g\rangle \subset {\rm Aut}(X)$ be the cyclic group which is generated by $g$, and $X/\langle g\rangle $ be
the quotient space of $X$ by $\langle g\rangle $.

Automorphisms of smooth plane curves of degree $d\geq 4$ are given by automorphisms of $\mathbb P^2$ ([\ref{bio:acgh}, Appendix A, 17 and 18)]).
When the degree $d$ is fixed,
the orders of automorphims of smooth plane curves of degree $d$ are calculated ([\ref{bio:bb2016},\ {\rm Theorem}\ 6]).
The structures of the automorphism groups are studied based on the orders of automorphisms ([\ref{bio:bb2016},\ref{bio:th21}]).
For degrees $4$, $5$, and $6$, complete classifications of the automorphism groups of smooth plane quartic, quintic, and sextic curves are provided in [\ref{bio:hen}, \ref{bio:kuko}], [\ref{bio:bb}], and [\ref{bio:bb25}], respectively.
Further work extends this to $d \geq 4$, presenting a list of potential the automorphism group structures, some of which are described using exact group sequences in [\ref{bio:haru}].
However, the realization of these structures as the automorphism groups for specific degrees, particularly for fixed $d$, is not fully clarified in certain cases.

Let $C_{a,b}\subset \pp$ be a smooth curve of bidegree $(a,b)$.
When $a,b\geq 3$, automorphisms of $C_{a,b}$ are given by automorphisms of $\pp$
([\ref{bio:tt12}, {\rm Theorem}\,3]).
For the case where the gonality of smooth curves in $\pp$ given by Galois extensions,
the Galois groups are determined ([\ref{bio:tt12},\ {\rm Theorem}\,2]).
In this paper,
we calculate the orders of automorphisms of smooth curves in $\pp$. 
We provide a sufficient condition for the quotient space by an automorphism to be $\mathbb P^1$.
The following Theorem \ref{main} is our first main result.
\begin{thm}\label{main}
Let $C_{a,b}\subset \pp$ be a smooth curve of bidegree $(a,b)$, and let $f$ be an automorphism of $C_{a,b}$ where $a,b\geq 3$.
Then we have the following:
\begin{enumerate}
\item[$(i)$]${\rm ord}(f)$ divides either $6$, $k-2$, $2(k-1)$, $(a-1)(b-1)+1$, $a(b-1)$, $(a-1)b$, 
or $ab$ where $k\in\{a,b\}$.
\item[$(ii)$]If ${\rm ord}(f)$ is either $(a-1)(b-1)+1$, $a(b-1)$, $(a-1)b$, 
or $ab$,
then $C_{a,b}/\langle f\rangle $ is $\mathbb P^1$.
\item[$(iii)$]If $|{\rm Fix}(f)|>0$ and ${\rm ord}(f)$ divides $lm$ where $l\geq2$ and $m:={\rm Max}\{a,b\}$,
then $C_{a,b}/\langle f\rangle $ is $\mathbb P^1$.
\end{enumerate}
\end{thm}
For $a,b\geq 4$,
we identify smooth curves of bidegree $(a,b)$ in $\pp$ with automorphisms whose orders are certain large values.
Our second main result is the following Theorem \ref{main2}.
Here, let $\mathbb C^{\ast}:=\mathbb C\backslash\{0\}$.
\begin{thm}\label{main2}
Let $C_{a,b}\subset \pp$ be a smooth curve of bidegree $(a,b)$ where $a,b\geq 4$, and let $f$ be an automorphism of $C_{a,b}$.
\begin{enumerate}
\item[$(i)$]If ${\rm ord}(f)=ab$, then 
$C_{a,b}$ is isomorphic to the smooth plane curve defined by $X_0^aY_0^b+X_0^aY_1^b+X_1^aY_0^b+sX_1^aY_1^b=0$
for some $s \in \mathbb{C}^*$.
This family of smooth curves is parametrized by $s \in \mathbb{C}^*$.
\item[$(ii)$]If ${\rm ord}(f)=(a-1)b$, then 
$C_{a,b}$ is isomorphic to the smooth plane curve defined by $X_0^aY_0^b+X_0^{a-1}X_1Y_1^b+X_0X_1^{a-1}Y_0^b+sX_1^aY_1^b=0$
for some $s \in \mathbb{C}^*$.
This family of smooth curves is parametrized by $s \in \mathbb{C}^*$.
\item[$(iii)$]If ${\rm ord}(f)=a(b-1)$, then 
$C_{a,b}$ is isomorphic to the smooth plane curve defined by $X_0^aY_0^b+X_1^aY_0^{b-1}Y_1+X_0^aY_0Y_1^{b-1}+sX_1^aY_1^b=0$
for some $s \in \mathbb{C}^*$.
This family of smooth curves is parametrized by $s \in \mathbb{C}^*$.
\item[$(iv)$]If ${\rm ord}(f)=(a-1)(b-1)+1$,
then $C_{a,b}$ is isomorphic to the smooth plane curve defined by $X_0^{a-1}X_1Y_0^b+X_0^aY_0Y_1^{b-1}+sX_1^aY_0^{b-1}Y_1+
s'X_0X_1^{a-1}Y_1^b=0$, or
$X_0^aY_0^{b-1}Y_1+X_0X_1^{a-1}Y_0^b+sX_0^{a-1}X_1Y_1^b+s'X_1^aY_0Y_1^{b-1}=0$
for some $(s,s') \in (\mathbb{C}^*)^2$.
These families of smooth curves are parametrized by $(s,s') \in (\mathbb{C}^*)^2$.
\end{enumerate}
\end{thm}
Let $C_{a,b}\subset \pp$ be a smooth curve of bidegree $(a,b)$.
By the adjunction formula,
the genus of $C_{a,b}$ is $(a-1)(b-1)$ ([\ref{bio:harts}, Chapter V, Example 1.5.2]).
Let $C_d\subset \mathbb P^2$ be a smooth curve of degree $d$.
By the adjunction formula,
the genus of $C_d$ is $\frac{1}{2}(d-1)(d-2)$ ([\ref{bio:harts}, Chapter V, Example 1.5.1]).
There are many positive integers $a,b\geq 3$ such that $(a-1)(b-1)\not=\frac{1}{2}(d-1)(d-2)$ for any $d\in\mathbb N$, e.g. $a=2^k+1$ and $b=2^l+1$.
Therefore, our main results provide a new insight that differs from Theorems $1$ and $6$ in [\ref{bio:bb2016}].
Section 2 is preliminary. 
We explain Theorems $1$ and $6$ in [\ref{bio:bb2016}].
We present sufficient conditions for the quotient space of a smooth plane curve by its automorphism to be $\mathbb P^1$.
We prepare symbols and other things to prove the main results.
Let $C_{a,b}\subset \pp$ be a smooth curve of bidegree $(a,b)$ where $a,b\geq 3$, and let $f$ be an automorphism of $C_{a,b}$.
In Scetion 3, we calculate the order of $f$ based on the number of special points on $C_{a,b}$.
Afterward,
we show Theorems \ref{main} and \ref{main2}.
\section{Preliminary}
We represent the coordinate system of $\mathbb P^2$ with $[X:Y:Z])$. 
\begin{thm}\label{bb1}$([\ref{bio:bb2016},\ {\rm Theorem}\ 1\ {\rm and}\ {\rm Theorem}\ 6])$.
Let $C\subset \mathbb P^2$ be a smooth plane curve of degree $d\geq 4$, and let $f$ be an automorphism of $C$.
Then ${\rm ord}(f)$ divides either $(d-1)d$, $(d-1)^2$, $(d-2)d$, or $(d-2)(d-1)+1$.
In addition, we have the following:
\begin{enumerate}
\item[$(i)$]If ${\rm ord}(f)=(d-1)d$, then 
$C$ is isomorphic to the smooth plane curve defined by $X^d+Y^d+XZ^{d-1}=0$.
\item[$(ii)$]If ${\rm ord}(f)=(d-1)^2$, then 
$C$ is isomorphic to the smooth plane curve defined by $X^d+Y^{d-1}Z+XZ^{d-1}=0$.
\item[$(iii)$]If ${\rm ord}(f)=(d-2)d$, then 
$C$ is isomorphic to the smooth plane curve defined by $X^d+Y^{d-1}Z+YZ^{d-1}=0$.
\item[$(iv)$]If ${\rm ord}(f)=(d-2)(d-1)+1$, then 
$C$ is isomorphic to the smooth plane curve defined by $X^{d-1}Y+Y^{d-1}Z+XZ^{d-1}=0$.
\end{enumerate}
\end{thm}
Additionally, Badr and Bars in $[\ref{bio:bb2016},\ {\rm Theorem}\ 1]$
 determine the automorphism groups of the smooth plane curves defined by the equations 
$X^d+Y^d+XZ^{d-1}=0$, $X^d+Y^{d-1}Z+XZ^{d-1}=0$, $X^d+Y^{d-1}Z+YZ^{d-1}=0$, and $X^{d-1}Y+Y^{d-1}Z+XZ^{d-1}=0$.
\begin{thm}
$([\ref{bio:bb2016},\ {\rm Corollary}\ 24\ {\rm and}\ {\rm Corollary}\ 32],[\ref{bio:th23g},\ {\rm Theorem}\ 4.1])$.
Let $C\subset \mathbb P^2$ be a smooth plane curve of degree $d\geq 4$, and let $f$ be an automorphism of $C$.
If ${\rm ord}(f)$ is either $k(d-1)$, $kd$, or $(d-2)(d-1)+1$ where $k\geq 2$, then $C/\langle f\rangle \cong \mathbb P^1$.
\end{thm}
\begin{thm}$([\ref{bio:th21l},\ {\rm Theorem}\ 1.7])$.
Let $C\subset \mathbb P^2$ be a smooth plane curve of degree $d\geq 4$, and let $f$ be an automorphism of $C$.
\begin{enumerate}
\item[$(i)$]We assume that ${\rm ord}(f)=d-1$. Then $|{\rm Fix}(f)|\not= 2$ if and only if $C/\langle f\rangle \cong \mathbb P^1$.
\item[$(i)$]We assume that ${\rm ord}(f)=d$. Then $|{\rm Fix}(f)|\not=0$ if and only if $C/\langle f\rangle \cong \mathbb P^1$.
\end{enumerate}
\end{thm}
We represent the coordinate system of $\mathbb P^1\times \mathbb P^1$ with $([X_0:X_1],[Y_0:Y_1])$. 
We set 
\begin{equation*}
\begin{split}
Q_1:=([1:0],[1:0]),\ \ \ \ \ \ \ &Q_2:=([1:0],[0:1]),\\
Q_3:=([0:1],[1:0]),\ \ \ \ \ \ \ &Q_4:=([0:1],[0:1]).
\end{split}
\end{equation*}
Let $C_{a,b}\subset\pp$ be a smooth curve of bidegree $(a,b)$, and let $\ff$ be the defining equation of $C_{a,b}$.
Since the bidegree of $C_{a,b}$ is $(a,b)$, there are $s_{i,j}\in\mathbb C$ for $0\leq i\leq a$ and $0\leq j\leq b$ such that
\[F(X_0,X_1,Y_0,Y_1)=\sum_{\substack{0\leq i\leq a\\0\leq j\leq b}}s_{i,j}X_0^iX_1^{a-i}Y_0^jY_1^{b-j}.\]
For the above representation, we define the bihomogeneous polynomials $\fa$, $\fb$, $\fc$, and $\fd$ as follows:
\begin{equation*}
\begin{split}
\fa:=&s_{a,b}X_0^aY_0^b+s_{a-1,b}X_0^{a-1}X_1Y_0^b+s_{a,b-1}X_0^aY_0^{b-1}Y_1,\\
\fb:=&s_{a,0}X_0^aY_1^b+s_{a-1,0}X_0^{a-1}X_1Y_1^b+s_{a,1}X_0^aY_0Y_1^{b-1},\\
\fc:=&s_{0,b}X_1^aY_0^b+s_{1,b}X_0X_1^{a-1}Y_0^b+s_{0,b-1}X_1^aY_0^{b-1}Y_1,\\
\fd:=&s_{0,0}X_1^aY_1^b+s_{1,0}X_0X_1^{a-1}Y_1^b+s_{0,1}X_1^aY_0Y_1^{b-1}.
\end{split}
\end{equation*}
We define sets $\mathbb E$ and $\mI$ of specific pairs of integers $(i,j)$ as follows:
\begin{equation*}
\begin{split}
\mathbb E&:=
\left\{
\begin{aligned}
&(a,b),\,(a-1,b),\,(a,b-1),\,(a,0),\,(a-1,0),\,(a,1),\\
&(0,b),\,(1,b),\,(0,b-1),\,(0,0),\,(1,0),\,(0,1)
\end{aligned}
\right\},\\
\mI&:=\{(i,j)\,|\,0\leq i\leq a,\ 0\leq j\leq b,\ {\rm and}\ (i,j)\not\in \mathbb E\}. 
\end{split}
\end{equation*}
Note that
\begin{equation}\label{01}
\begin{split}
\sum_{i=1}^4F_{Q_i}&=\sum_{(i,j)\in \mathbb E}s_{i,j}X_0^iX_1^{a-i}Y_0^jY_1^{b-j},
\end{split}
\end{equation}
\begin{equation}\label{02}
\begin{split}
F(X_0,X_1,Y_0,Y_1)&=\sum_{i=1}^4F_{Q_i}+\sum_{(i,j)\in\mI}s_{i,j}X_0^iX_1^{a-i}Y_0^jY_1^{b-j}.
\end{split}
\end{equation}
The smoothness of $C_{a,b}$ at $Q_1,\ldots, Q_4$ implies the following Lemma \ref{1} for $\fa$, $\fb$, $\fc$, and $\fd$.
\begin{lem}\label{1}
Let $C_{a,b}\subset\pp$ be a smooth curve of bidegree $(a,b)$ with $a,b\geq 1$.
Let $\ff$ be the defining equation of $C_{a,b}$, and let $\fa$, $\fb$, $\fc$, and $\fd$ be the bihomogeneous polynomials defined as above.
Then $\fa$, $\fb$, $\fc$, and $\fd$ are non-zero polynomials.
Moreover, $Q_1\in C_{a,b}$ $($similarly $Q_2\in C_{a,b},\ Q_3\in C_{a,b},\ Q_4\in C_{a,b})$ if and only if $s_{a,b}=0$ $($similarly $s_{a,0}=0,\ s_{0,b}=0,\ s_{0,0}=0)$.
\end{lem}
\begin{proof}
We show that $\fa\not=0$ and that  $Q_1\in C_{a,b}$ if and only if $s_{a,b}=0$.
The other cases follow by a similar argument.
We set 
\[F(X_0,X_1,Y_0,Y_1)=\sum_{\substack{0\leq i\leq a\\0\leq j\leq b}}s_{i,j}X_0^iX_1^{a-i}Y_0^jY_1^{b-j}\]
where $s_{i,j}\in\mathbb C$ for $0\leq i\leq a$ and $0\leq j\leq b$.
Then 
\[\fa=s_{a,b}X_0^aY_0^b+s_{a-1,b}X_0^{a-1}X_1Y_0^b+s_{a,b-1}X_0^aY_0^{b-1}Y_1.\]
We set 
\[x:=\frac{X_1}{X_0},\ \ y:=\frac{Y_1}{Y_0},\ \  {\rm and}\ \ \ f(x,y):=F(1,\frac{X_1}{X_0},1,\frac{Y_1}{Y_0}).\]
Then
\[f(x,y)=s_{a,b}+s_{a-1,b}x+s_{a,b-1}y+\sum_{(i,j)\in\mI}s_{i,j}x^{a-i}y^{b-j}. \]
If $\fa=0$, i.e. $s_{a,b}=s_{a-1,b}=s_{a,b-1}=0$, then the affine curve defined by $f(x,y)=0$ is singular at $(0,0)\in\mathbb C^2$.
This implies that $C_{a,b}$ is singular at $Q_1\in\pp$.
This is a contradiction.
Therefore, $\fa\not=0$.
In addition, by the equation $f(x,y)$, we see that $Q_1\in C_{a,b}$ if and only if $s_{a,b}=0$.
\end{proof}
\begin{thm}$([\ref{bio:tt12},\ {\rm Theorem}\ 3])$.
Let \( X_e \) be the Hirzebruch surface of degree \( e \geq 0 \), and let \( C_0 \) be a section with self-intersection \( C_0^2 = -e \), and \( f \) a fiber of the ruling \( X_e \to \mathbb{P}^1 \).  
Let \( C \sim aC_0 + b f \) be a nonsingular projective curve on \( X_e \), and assume that \( a, b \geq 3 \) and \( (e, a) \ne (1, b), (1, b - 1) \).  
Then for every \( f \in \mathrm{Aut}(C) \), there exists \( \widehat{f} \in \mathrm{Aut}(X_e) \) such that \( \widehat{f}(C) = C \) and \( \widehat{f}_{|C} = f \).
\end{thm}

Let ${\rm GL}(2,\mathbb C)$ be the general linear group of $\mathbb C^2$,
let $I_2$ be the identity matrix of size $2$, and let $Z(\mathbb C^2):=\{A\in{\rm GL}(2,\mathbb
	C)\,|\,A=aI_2\ {\rm for\ some}\ a\in\mathbb C^{\ast}\}$ be the center of ${\rm GL}(2,\mathbb C)$.
	The projective linear group ${\rm PGL}(2,\mathbb C)$ is the quotient group
	${\rm GL}(2,\mathbb C)/Z(\mathbb C^2)$.
	Let $q\co{\rm GL}(2,\mathbb C)\ra{\rm PGL}(2,\mathbb C)$ be the quotient morphism.
For $A\in{\rm GL}(2,\mathbb C)$, we set $[A]:=q(A)\in{\rm PGL}(2,\mathbb C)$.
For $a,b\in\mathbb C^{\ast}$, we set 
\[D(a,b):=\begin{pmatrix}
a&0\\
0&b
\end{pmatrix}\in{\rm GL}(2,\mathbb C).\]

Let $C_{a,b}\subset \pp$ be a smooth curve of bidegree $(a,b)$, and
let $f$ be an automorphism of $C_{a,b}$.
For $[A]\times[B]\in {\rm PGL}(2,\mathbb C)\times {\rm PGL}(2,\mathbb C)\subset {\rm Aut}(\pp)$ where $A,B\in{\rm GL}(2,\mathbb C)$,
if the restriction of $[A]\times [B]$ to $C_{a,b}$ is equal to $f$,
then we denote this by writing
\[f=[A]\times[B].\]
If either $A$ or $B$ is $I_2$,
we determine the order in Lemma \ref{2}.
For the case where $A\neq I_2$ and $B\neq I_2$,
we calculate the orders of automorphisms based on the number of $|C_{a,b}\cap \{Q_i\}_i^4|$.
Additionally, in Lemma \ref{3}, we show that $|C_{a,b}\cap \{Q_i\}_{i=1}^4|\not=1$.

For a positive integer $n$, let $e_n$ be a primitive $n$-th root of unity.
\begin{lem}\label{-1}
Let $C_{a,b}\subset \pp$ be a smooth curve of bidegree $(a,b)$ with an automorphism $f=[A]\times [B]$ where $a,b\geq 3$ and $A,B\in{\rm GL}(2,\mathbb C)$.
By replacing the coordinate systems of the first and second components of $\pp$ individually if necessary,
we get $f=[D(e_n,1)]\times [D(e_m,1)]$ where $n$ are $m$ positive integers.
\end{lem}
\begin{proof}
Since $a,b\geq 3$, the genus $(a-1)(b-1)$ of $C_{a,b}$ is $4$ or more.
Therefore, ${\rm Aut}(C_{a,b})$ is a finite group.
This implies that ${\rm ord}(f)$ is finite.
Let $l:={\rm ord}(f)$.
Then $A^l=\lambda_1I_2$ and $B^l=\lambda_2I_2$ where $\lambda_1,\lambda_2\in\mathbb C^{\ast}$.
This shows that $A$ and $B$ are diagonalizable.
By exchanging the coordinate systems of the first and second components of $\pp$ individually if necessary,
$A$ and $B$ are diagonal matrices $D(a_1,b_1)$ and $D(a_2,b_2)$, respectively, with $a_i, b_i \in \mathbb{C}^*$ satisfying $a_i^l = b_i^l =\lambda_i$ for $i = 1,2$.
Thus, $f=[D(a_1,b_1)]\times[D(a_2,b_2)]$.
Let $s_i$ be the minimal positive integer such that $a_i^{s_i}=b_i^{s_i}=\lambda_i$ for $i=1,2$.
Then $\frac{a_i}{b_i}$ is a primitive $s_i$-th root of unity.
Since $[D(a_i,b_i)]=[D(e_{s_i},1)]$ for $i=1,2$, we conclude that $f=[D(e_{s_1},1)]\times[D(e_{s_2},1)]$.
\end{proof}
Observe that for $f=[D(e_n,1)]\times [D(e_m,1)]$ where $n,m\geq2$, ${\rm Fix}(f)$ is contained in $\{Q_i\}_{i=1}^4$.

Let $\mathcal O_{\mathbb P^1}(-1)$ be the tautological line bundle on $\mathbb P^1$, and let $\mathcal O_{\mathbb P^1}(1)$ be its dual line bundle. The Picard group ${\rm Pic}(\mathbb P^1)$ is isomorphic to $\mathbb{Z}$, generated by the class of the line bundle $\mathcal{O}_{\mathbb{P}^1}(1)$.
Let $p_i\co \pp\ra\mathbb P^1$ be the $i$-th projection for $i=1,2$.
The Picard group \({\rm Pic}(\pp)\) is isomorphic to \(\mathbb{Z} \oplus \mathbb{Z}\), generated by the classes of the line bundles \( p_1^*\mathcal{O}_{\mathbb{P}^1}(1) \) and \( p_2^*\mathcal{O}_{\mathbb{P}^1}(1) \).
The self-intersection number of the line bundle 
\( p_1^*\mathcal{O}_{\mathbb{P}^1}(a) \otimes p_2^*\mathcal{O}_{\mathbb{P}^1}(b) \) is 
$2ab$ for $a,b\in \mathbb Z$. 

Let $\s$ denote the automorphism of $\pp$ that exchanges the first and second components of $\pp$.
The following result is well known among experts, but we include a proof here for the reader's convenience.
\begin{thm}
For $g\in {\rm Aut}(\pp)$, there are automorphisms $g_1,g_2\in {\rm Aut}(\mathbb P^1)$ such that 
$g=g_1\times g_2$ or $g=\s\circ (g_1\times g_2)$.
\end{thm}
\begin{proof}
Since the self-intersection number of $p_i^*\mathcal{O}_{\mathbb{P}^1}(1)$ is $0$ for $i=1,2$,
 the pullback \(g^*\) either fixes or interchanges the generators \(p_1^*\mathcal{O}_{\mathbb{P}^1}(1)\) and \(p_2^*\mathcal{O}_{\mathbb{P}^1}(1)\) in \({\rm Pic}(\pp)\).
If \(g^*\) fixes the generators \(p_1^*\mathcal{O}_{\mathbb{P}^1}(1)\) and \(p_2^*\mathcal{O}_{\mathbb{P}^1}(1)\) in \({\rm Pic}(\pp)\), then 
there are automorphisms $g_1,g_2\in {\rm Aut}(\mathbb P^1)$ such that $g_i\circ p_i=p_i\circ g$ for $i=1,2$.
Thus, $g=g_1\times g_2$.
If \(g^*\) interchanges the generators \(p_1^*\mathcal{O}_{\mathbb{P}^1}(1)\) and \(p_2^*\mathcal{O}_{\mathbb{P}^1}(1)\) in \({\rm Pic}(\pp)\), then $(s_(1,2)\circ g^*)$ fixes them.
Thus, there are automorphisms $g_1,g_2\in {\rm Aut}(\mathbb P^1)$ such that $\s\circ g=g_1\times g_2$, and hence $g=\s\circ (g_1\times g_2)$.
\end{proof}
Note that $\bigl(\s\circ (g_1\times g_2)\bigr)^2=(g_2\circ g_1)\times(g_1\circ g_2)$.
\begin{lem}\label{8}
Let $C_{a,b}\subset \pp$ be a smooth curve of bidegree $(a,b)$ with an automorphism $f=s_{(1,2)}\circ([A]\times[B])$ where $A,B\in{\rm GL}(2,\mathbb C)$. Then the following hold:
\begin{enumerate}
\item[$(i)$]We have $a=b$.
\item[$(ii)$]If $n:={\rm ord}(f^2)\geq 2$, then by replacing the coordinate systems of the first and second components of $\pp$ individually if necessary,
we get that $f=s_{(1,2)}\circ([A']\times[B'])$ such that $f^2=[D(e_n,1)]\times [D(e_n,1)]$, and $A'$ and $B'$ are diagonal matrices where $A'B' = B'A' = D(e_n, 1)$.
\end{enumerate}
\end{lem}
\begin{proof}
Since $\mathcal O_{\pp}(C_{a,b})=p_1^{\ast}\mathcal O_{\mathbb P^1}(a)\otimes p_2^{\ast}\mathcal O_{\mathbb P^1}(b)$ in ${\rm Pic}(\pp)$ and $f(C_{a,b})=C_{a,b}$, 
we get that $f^{\ast}(p_1^{\ast}\mathcal O_{\mathbb P^1}(a)\otimes p_2^{\ast}\mathcal O_{\mathbb P^1}(b))=p_1^{\ast}\mathcal O_{\mathbb P^1}(a)\otimes p_2^{\ast}\mathcal O_{\mathbb P^1}(b)$ in ${\rm Pic}(\pp)$.
By the presence of \(s_{(12)}\) in the definition of \(f\), the pullback \(f^*\) interchanges the generators \(p_1^*\mathcal{O}_{\mathbb{P}^1}(1)\) and \(p_2^*\mathcal{O}_{\mathbb{P}^1}(1)\).
Thus,  $f^{\ast}(p_1^{\ast}\mathcal O_{\mathbb P^1}(a)\otimes p_2^{\ast}\mathcal O_{\mathbb P^1}(b))=p_1^{\ast}\mathcal O_{\mathbb P^1}(b)\otimes p_2^{\ast}\mathcal O_{\mathbb P^1}(a)$ in ${\rm Pic}(\pp)$.
Then we have $p_1^{\ast}\mathcal O_{\mathbb P^1}(a)\otimes p_2^{\ast}\mathcal O_{\mathbb P^1}(b))=p_1^{\ast}\mathcal O_{\mathbb P^1}(b)\otimes p_2^{\ast}\mathcal O_{\mathbb P^1}(a)$ in ${\rm Pic}(\pp)$.
Therefore, we have $a=b$.

We assume that $n\geq 2$.
Since $f^2=[BA]\times [AB]$, and ${\rm ord}(f^2)=n<\infty$, by multiplying by a constant if necessary, we may assume that the eigenvalues of $AB$ are $1$ and $e_n$.
Note that $AB$ and $BA$ have the same eigenvalues.
Let $S,T\in{\rm GL}(2,\mathbb C)$ be matrices such that 
\[S(BA)S^{-1}=T(AB)T^{-1}=D(e_n,1).\]
We set $k:=[S]\times [T]\in{\rm Aut}(\pp)$ and $f':=k\circ f\circ k^{-1}\in{\rm Aut}(\pp)$.
Then 
\[f'=\s\circ([TAS^{-1}]\times[SBT^{-1}])\]
and
\begin{equation*}
\begin{split}
f'\circ f'&=[SBAS^{-1}]\times[TABT^{-1}]\\
&=[D(e_n,1)]\times [D(e_n,1)].
\end{split}
\end{equation*}
Therefore,	by replacing the coordinate system if necessary,
$f=\s\circ ([A']\times [B'])$ where $B'A'=A'B'=D(e_n,1)$.
Since $B'A'=A'B'=D(e_n,1)$, 
\[A'D(e_n,1)=D(e_n,1)A'\ \ {\rm and}\ \ B'D(e_n,1)=D(e_n,1)B'.\]
Since $n\geq 2$ and $A',B'\in{\rm GL}(2,\mathbb C)$, by the above equations, we get that $A'$ and $B'$ are diagonal matrices.
\end{proof}

Here, we introduce theorem from [\ref{bio:tt12}] concerning curves in $\pp$.
\begin{thm}\label{tt12}$([\ref{bio:tt12},\ {\rm Theorem}\ 2 {\rm and\ Remark}\ 4])$.
Let $C_{a,b}$ be a smooth curve of bidegree $(a,b)$ in $\pp$ where $a,b\geq 3$.
We assume that the morphism ${p_2}_{|C_{a,b}}$ is Galois.
Then the Galois group of ${p_2}_{|C_{a,b}}$ is a subgroup of ${\rm Aut}(\mathbb{P}^1_{(1)})$, where $\mathbb{P}^1_{(1)}$ denotes the first factor of $\mathbb{P}^1 \times \mathbb{P}^1$.
The following assertions then hold:
\begin{enumerate}
\item[$(1)$]The Galois group of ${p_2}_{|C_{a,b}}$ is one of the following:
\begin{enumerate}
\item[$(a)$]a cyclic group $C_a$ of order $a$; 
\item[$(b)$]a dihedral group $D_{2m}$ of order $2m=a$;
\item[$(c)$]the alternating group $A_4$ on four letters;
\item[$(d)$]the symmetric group $S_4$ on four letters;
\item[$(e)$]the alternating group $A_5$ on five letters.
\end{enumerate}
\item[$(2)$]
By choosing suitable coordinate system, we can express $C_{a,b}$ up to isomorphism as follows:
\begin{enumerate}
\item[$(a)$]$F(Y_0,Y_1)X_0^a-G(Y_0,Y_1)X_1^a=0$;
\item[$(b)$]$4F(Y_0,Y_1)X_0^mX_1^m+G(Y_0,Y_1)(X_0^m-X_1^m)^2=0$, where $a=2m$;
\item[$(c)$]$F(Y_0,Y_1)(X_0^4-2\sqrt{3}X_0^2X_1^2-X_1^4)^3-G(Y_0,Y_1)(X_0^4+2\sqrt{3}X_0^2X_1^2-X_1^4)^3=0$;
\item[$(d)$]$108F(Y_0,Y_1)X_0^4X_1^4(X_0^4-X_1)^4-G(Y_0, Y_1)(X_0^8+14X_0^4X_1^4+X_0^8)^3=0$;
\item[$(e)$]$1728F(Y_0,Y_1)X_0^5X_1^5(X_0^{10}+11X_0^5X_1^5-X_1^{10})^5\\
+G(Y_0,Y_1)(X_0^{20}-228X_0^{15}X_1^5+494X_0^{10}X_1^{10}+228X_0^5X_1^{15}+X_1^{20})^3=0$.
\end{enumerate}
Here, $F(Y_0,Y_1)$ and $G(Y_0,Y_1)$ are homogeneous polynomials of degree $b$.
Each case corresponds to both the group in $(1)$ and the equation in $(2)$ with the same label.
\end{enumerate}
\end{thm}
The Galois coverings of rational curves, where the covering space is a non-singular curve on the projective plane, have been studied in $[\ref{bio:th23g},\ref{bio:my00},\ref{bio:y01f}]$.
Moreover, methods for determining whether such a curve is a cyclic covering of a rational curve based on its automorphisms are provided $([\ref{bio:bb2016},\ref{bio:th21l}])$.
The following lemma rephrases the case $(a)$ of Theorem \ref{tt12} in terms of automorphisms.
\begin{lem}\label{th}
Let $C_{a,b}\subset \pp$ be a smooth curve of bidegree $(a,b)$.
We assume that $C_{a,b}$ has an automorphism $f$ such that $f^k=[D(e_a,1)]\times [I_2]$ for some $k\geq 1$.
Then $C_{a,b}/\langle f\rangle \cong \mathbb P^1$.
\end{lem}
\begin{proof}
For the morphism ${p_2}_{|C_{a,b}}\co C_{a,b}\ra\mathbb P^1$,
the automorphism $f^k=[D(e_a,1)]\times [I_2]$ satisfies ${p_2}_{|C_{a,b}}={p_2}_{|C_{a,b}}\circ f^k$.
Since the bidegree of $C_{a,b}$ is $(a,b)$,
the degree of ${p_2}_{|C_{a,b}}$ is $a$.
Thus, ${p_2}_{|C_{a,b}}\co C_{a,b}\ra\mathbb P^1$ is a Galois morphism with Galois group $\langle f^k\rangle $.
Consequently, we have $C_{a,b}/\langle f^k\rangle \cong\mathbb P^1$, and therefore $C_{a,b}/\langle f\rangle \cong\mathbb P^1$.
\end{proof}

\section{Proof of main theorems}
For $d\geq 0$, let $\mathbb C[Y_0,Y_1]_d\subset\mathbb C[Y_0,Y_1]$ be the vector space of forms of degree $d$.
\begin{lem}\label{2}
Let $C_{a,b}\subset \pp$ be a smooth curve of bidegree $(a,b)$ with
an automorphism $f$ such that $f=[D(e_n,1)]\times [I_2]$.
Then ${\rm ord}(f)$ divides $a$.
\end{lem}
\begin{proof}
Let $\ff$ be the defining equation of $C_{a,b}$.
We set 
\[F(X_0,X_1,Y_0,Y_1)=\sum_{0\leq i\leq a}X_0^iX_1^{a-i}F_i(Y_0,Y_1)\]
where $F_i(Y_0,Y_1)\in\mathbb C[Y_0,Y_1]_b$ for $0\leq i\leq a$.
Since $C_{a,b}$ is smooth, it is in particular integral.
Thus, $\ff$ is irreducible.
If $F_0(Y_0,Y_1)=0$ (resp. $F_a(Y_0,Y_1)=0$), then $\ff$ is divisible by $X_1$ (resp. $X_0$), which contradicts the irreducibility of $\ff$.
Therefore, both $F_0(Y_0,Y_1)$ and $F_a(Y_0,Y_1)$ are non-zero polynomials.

Since $f=[D(e_n,1)]\times [I_2]$ is an automorphism of $C_{a,b}$,
	$f^{\ast}\ff=t\ff$ for some $t\in\mathbb C^{\ast}$, i.e.
	\[\sum_{0\leq i\leq a}e_n^iX_0^iX_1^{a-i}F_i(Y_0,Y_1)=\sum_{0\leq i\leq a}tX_0^iX_1^{a-i}F_i(Y_0,Y_1).\]
	Since $F_0(Y_0,Y_1)\not=0$ and $F_a(Y_0,Y_1)\not=0$, we get that $t=e_n^0=e_n^a$, and hence $e_n^a=1$. 
	Therefore, the order of $f$ divides $a$.
\end{proof}
\begin{lem}\label{3}
Let $C_{a,b}\subset \pp$ be a smooth curve of bidegree $(a,b)$ with an automorphism $f$ such that $f=[D(e_n,1)]\times [D(e_m,1)]$ where $n,m\geq 2$.
Then $|\{Q_i\}_{i=1}^4 \cap C_{a,b}|\not=1$. 
\end{lem}
\begin{proof}
We assume that $|\{Q_i\}_{i=1}^4 \cap C_{a,b}|=1$, and for simplicity, let $\{Q_i\}_{i=1}^4 \cap C_{a,b}=\{Q_4\}$.
Let $\ff$ be the defining equation of $C_{a,b}$.
By Lemma \ref{1}, the polynomial $F_{Q_1}$ contains the monomial $X_0^a Y_0^b$, $F_{Q_2}$ contains $X_0^a Y_1^b$, and $F_{Q_3}$ contains $X_1^a Y_0^b$, whereas $F_{Q_4}$ does not contain the monomial $X_1^a Y_1^b$. Therefore, let
\[
F :=s_{a,b} X_0^a Y_0^b + s_{a,0} X_0^a Y_1^b + s_{0,b} X_1^a Y_0^b + s_{1,0} X_0 X_1^{a-1} Y_1^b + s_{0,1} X_1^a Y_0 Y_1^{b-1}
\]
where $s_{a,b}, s_{a,0}, s_{0,b} \in \mathbb{C}^\ast$ and $s_{1,0}, s_{0,1} \in \mathbb{C}$ with $(s_{1,0}, s_{0,1}) \neq (0,0)$.
Then, by the equation~\eqref{01}, 
\[
\sum_{i=1}^4F_{Q_i} = F + \sum_{(i,j) \in \mathbb{E} \setminus \{(a,b), (a,0), (0,b), (1,0), (0,1)\}} s_{i,j} X_0^i X_1^{a-i} Y_0^j Y_1^{b-j},
\]
where $s_{i,j} \in \mathbb{C}$ for $(i,j) \in \mathbb{E} \setminus \{(a,b), (a,0), (0,b), (1,0), (0,1)\}$.
Since $f$ is an automorphism of $C_{a,b}$, we have
$f^{\ast}\ff=t\ff$ for some $t\in\mathbb C^{\ast}$.
Thus, by the equation~\eqref{02},
\[f^{\ast}F=tF,\]
which leads to the equation:
\begin{dmath*}
e_n^ae_m^bs_{a,b}X_0^aY_0^b+e_n^as_{a,0}X_0^aY_1^b+e_m^bs_{0,b}X_1^aY_0^b+e_ns_{1,0}X_0X_1^{a-1}Y_1^b+e_ms_{0,1}X_1^aY_0Y_1^{b-1}\\
=t(s_{a,b}X_0^aY_0^b+s_{a,0}X_0^aY_1^b+s_{0,b}X_1^aY_0^b+s_{1,0}X_0X_1^{a-1}Y_1^b+s_{0,1}X_1^aY_0Y_1^{b-1}).
\end{dmath*}
By comparing the coefficients of the monomials $X_0^a Y_0^b$, $X_0^a Y_1^b$, $X_1^a Y_0^b$, $X_0 X_1^{a-1} Y_1^b$, and $X_1^a Y_0 Y_1^{b-1}$, and using the assumptions that $s_{a,b}, s_{a,0}, s_{0,b} \in \mathbb{C}^\ast$ and $(s_{1,0}, s_{0,1}) \neq (0,0)$, we obtain the following equations:
\[e_n^ae_m^b=t,\ \ \ \ e_n^a=t,\ \ \ \ e_m^b=t,\ \ \ \ (e_n-t)(e_m-t)=0.\]
From the first three equations, we obtain $t=t^2$.
Since $t\in\mathbb C^{\ast}$, $t=1$.
The fourth equation implies $e_n=1$ or $e_m=1$.
This contradicts that $n,m\geq2$.
Therefore, $|\{Q_i\}_{i=1}^4 \cap C_{a,b}|\not=1$.
\end{proof}
We consider the case where $|\{Q_i\}_{i=1}^4 \cap C_{a,b}|=0$.
For integers $k,l\geq 1$, 
let ${\rm lcm}(k,l)$ be the least common multiple of $k$ and $l$.
\begin{pro}\label{4-}
Let $C_{a,b}\subset \pp$ be a smooth curve of bidegree $(a,b)$.
We assume that $|\{Q_i\}_{i=1}^4 \cap C_{a,b}|=0$, and $C_{a,b}$ has an automorphism $f=[D(e_n,1)]\times [D(e_m,1)]$
where $n,m\geq 2$.
Then ${\rm ord}(f)$ divides ${\rm lcm}(a,b)$, $e_n^a=e_m^b=1$, and $F_Q=s_{a,b} X_0^a Y_0^b + s_{a,0} X_0^a Y_1^b + s_{0,b} X_1^a Y_0^b + s_{0,0} X_1^aY_1^{b}$
where $s_{a,b}, s_{a,0}, s_{0,b},s_{0,0} \in \mathbb{C}^\ast$.
\end{pro}
\begin{proof}
Let $\ff$ be the defining equation of $C_{a,b}$.
By Lemma \ref{1}, the polynomial $F_{Q_1}$ contains the monomial $X_0^a Y_0^b$, $F_{Q_2}$ contains $X_0^a Y_1^b$, $F_{Q_3}$ contains $X_1^a Y_0^b$, and $F_{Q_4}$ contains the monomial $X_1^a Y_1^b$. Therefore, let
\[
F :=s_{a,b} X_0^a Y_0^b + s_{a,0} X_0^a Y_1^b + s_{0,b} X_1^a Y_0^b + s_{0,0} X_1^aY_1^{b}
\]
where $s_{a,b}, s_{a,0}, s_{0,b},s_{0,0} \in \mathbb{C}^\ast$.
Then, by the equation~\eqref{01}, 
\[
\sum_{i=1}^4F_{Q_i} = F + \sum_{(i,j) \in \mathbb{E} \setminus \{(a,b), (a,0), (0,b), (0,0)\}} s_{i,j} X_0^i X_1^{a-i} Y_0^j Y_1^{b-j},
\]
where $s_{i,j} \in \mathbb{C}$ for $(i,j) \in \mathbb{E} \setminus \{(a,b), (a,0), (0,b), (0,0)\}$.
Since $f$ is an automorphism of $C_{a,b}$, we have
$f^{\ast}\ff=t\ff$ for some $t\in\mathbb C^{\ast}$.
Thus, by the equation~\eqref{02}, $f^{\ast}F=tF$, i.e.
\begin{dmath*}
e_n^ae_m^bs_{a,b}X_0^aY_0^b+e_n^as_{a,0}X_0^aY_1^b+e_m^bs_{0,b}X_1^aY_0^b+s_{0,0}X_1^aY_1^b\\
=t(s_{a,b}X_0^aY_0^b+s_{a,0}X_0^aY_1^b+s_{0,b}X_1^aY_0^b+s_{0,0}X_1^aY_1^b).
\end{dmath*}
By comparing the coefficients of the monomials $X_0^a Y_0^b$, $X_0^a Y_1^b$, $X_1^a Y_0^b$, and $X_1^a Y_1^b$, and using the assumptions that $s_{a,b}, s_{a,0}, s_{0,b},s_{0,0} \in \mathbb{C}^\ast$, we obtain the following equations:
\[e_n^ae_m^b=t,\ \ \ e_n^a=t,\ \ \ e_m^b=t,\ \ \ 1=t.\]
From these equations, we get that $e_n^a=e_m^b=1$. As a result, ${\rm ord}(f)$ divides ${\rm lcm}(a,b)$.
Moreover, since $t=1$, we have $f^*F_Q=F_Q$.
Then $(e_n^ie_m^j-1)s_{i,j}=0$ for $(i,j) \in \mathbb{E} \setminus \{(a,b), (a,0), (0,b), (0,0)\}$.
Since $e_n^a=e_m^b=1$, $n,m\geq 2$, and 
\begin{equation*}
\begin{split}
\mathbb{E} \setminus \{(a,b), (a,0), (0,b), (0,0)\}&=
		\left\{
		\begin{aligned}
			&(a-1,b),\,(a,b-1),\,(a-1,0),\,(a,1),\\
			&(1,b),\,(0,b-1),\,(1,0),\,(0,1)
		\end{aligned}
		\right\},
	\end{split}
\end{equation*}
 we get that
$s_{i,j}=0$ for $(i,j) \in \mathbb{E} \setminus \{(a,b), (a,0), (0,b), (0,0)\}$.
Therefore, $F_Q=s_{a,b} X_0^a Y_0^b + s_{a,0} X_0^a Y_1^b + s_{0,b} X_1^a Y_0^b + s_{0,0} X_1^aY_1^{b}$.
\end{proof}
\begin{thm}\label{4}
Let $C_{a,b}\subset \pp$ be a smooth curve of bidegree $(a,b)$ with an automorphism $f$.
We assume that $|\{Q_i\}_{i=1}^4 \cap C_{a,b}|=0$.
\begin{enumerate}
\item[$(i)$]If $f=[D(e_n,1)]\times [D(e_m,1)]$ where $n,m\geq 2$, then ${\rm ord}(f)$ divides $ab$, and ${\rm Fix}(f)=\emptyset$.
\item[$(ii)$]If $f=s_{(1,2)}\circ ([A]\times [B])$ where $A,B\in{\rm GL}(2,\mathbb C)$ are diagonal matrices such that $AB=BA=D(e_n,1)$ with $n\geq 2$,
then $a=b$, and ${\rm ord}(f)$ divides $2a$, and ${\rm Fix}(f)=\emptyset$.
\end{enumerate}
\end{thm}
\begin{proof}
First, we assume that $f=[D(e_n,1)]\times [D(e_m),1]$ where $n,m\geq 2$.
By Proposition \ref{4-}, ${\rm ord}(f)$ divides ${\rm lcm}(a,b)$.
Thus, ${\rm ord}(f)$ divides $ab$.
Since $Q_i\not\in C_{a,b}$ for $i=1,\ldots,4$, we have ${\rm Fix}(f)=\emptyset$.
Next, we assume that $f=s_{(1,2)}\circ ([A]\times [B])$ where $A,B\in{\rm PGL}(2,\mathbb C)$ such that $A$ and $B$ are diagonal matrices, and $AB=BA=D(e_n,1)$ with $n\geq2$.
By Lemma \ref{8}, it follows that $a=b$.
By Proposition \ref{4-}, ${\rm ord}(f^2)$ divides $a$.
Then ${\rm ord}(f)$ divides $2a$.
Since $f^2=[D(e_n,1)]\times [D(e_n,1)]$ where $n\geq 2$, we have ${\rm Fix}(f^2)=\emptyset$.
Thus, ${\rm Fix}(f)=\emptyset$.
\end{proof}
Next, we consider the case where $|\{Q_i\}_{i=1}^4 \cap C_{a,b}|=2$.
\begin{lem}\label{7}
Let $C\subset \pp$ be a smooth curve with an automorphism $f$
such that $|\{Q_i\}_{i=1}^4 \cap C|=2$.
\begin{enumerate}
\item[$(i)$]We assume that $f=[D(e_n,1)]\times [D(e_m,1)]$ where $n,m\geq2$. 
By replacing the coordinate system if necessary, 
we get that $\{Q_i\}_{i=1}^4 \cap C$ is $\{Q_2,Q_3\}$ or $\{Q_3,Q_4\}$, and $f=[D(e_s^{\epsilon},1)]\times [D(e_t^{\delta},1)]$ where $\{s,t\}=\{n,m\}$ and $\epsilon,\delta\in\{{\pm 1}\}$. 
\item[$(ii)$]We assume that $f=s_{(1,2)}\circ ([A]\times [B])$
where $A$ and $B$ are diagonal matrices such that $AB=BA=D(e_n,1)$ with $n\geq2$. 
Then $\{Q_i\}_{i=1}^4 \cap C$ is $\{Q_1,Q_4\}$ or $\{Q_2,Q_3\}$.
\end{enumerate}
\end{lem}
\begin{proof}
First, we assume that $f=[D(e_n,1)]\times [D(e_m,1)]$ where $n,m\geq2$.
If $\{Q_i\}_{i=1}^4 \cap C=\{Q_1,Q_2\}$ $($resp. $\{Q_1,Q_4\})$, then by exchanging $X_0$ and $X_1$ in the first component of $\pp$, 
we obtain $\{Q_i\}_{i=1}^4 \cap C=\{Q_3,Q_4\}$ $($resp. $\{Q_2,Q_3\})$.
In these cases, 
$f=[D(e_n^{-1},1)]\times[D(e_m,1)]$.

If $\{Q_i\}_{i=1}^4 \cap C=\{Q_2,Q_4\}$, then by exchanging the first and second components of 
$\pp$, we obtain 
$\{Q_i\}_{i=1}^4 \cap C=\{Q_3,Q_4\}$.
In this in case,
$f=[D(e_m,1)]\times[D(e_n,1)]$.

If $\{Q_i\}_{i=1}^4 \cap C=\{Q_1,Q_3\}$, then by exchanging $Y_0$ and $Y_1$ in the second component of 
$\pp$,
followed by exchanging the first and second components of 
$\pp$,
we obtain 
$\{Q_i\}_{i=1}^4 \cap C=\{Q_3,Q_4\}$.
In this case,
$f=[D(e_m^{-1},1)]\times[D(e_n,1)]$.

Next, we assume that $f=s_{(1,2)}\circ ([A]\times [B])$
where $A$ and $B$ are diagonal matrices such that $AB=BA=D(e_n,1)$ with $n\geq2$. 
Since $A$ and $B$ are diagonal matrices,
$[A]\times[B](Q_i)=Q_i$ for $i=1,\ldots, 4$.
Since $\s(Q_1)=Q_1$, $\s(Q_4)=Q_4$, $\s(Q_2)=Q_3$, and because $|\{Q_i\}_{i=1}^4 \cap C|=2$, 
we get $\{Q_i\}_{i=1}^4 \cap C$ is either $\{Q_1,Q_4\}$ or $\{Q_2,Q_3\}$.
\end{proof}
In Lemma \ref{7},
by replacing the coordinate system, 
an automorphism may involve
$e_n^{-1}$, but note that $e_n^{-1}$ remains a primitive $n$-th root of unity.
In subsequent propositions and theorems, when calculating the order of an automorphism, 
it is sufficient that the diagonal entries of the matrix are primitive $n$-th roots of unity.
\begin{pro}\label{8.1}
Let $C_{a,b}\subset \pp$ be a smooth curve of bidegree $(a,b)$ with an automorphism $f=[D(e_n,1)]\times [D(e_m,1)]$ where $n,m\geq2$.
We assume that 
$\{Q_i\}_{i=1}^4 \cap C_{a,b}=\{Q_3,Q_4\}$.
Then one of the following holds:
\begin{enumerate}
\item[$(i)$]$e_n^{a-1}=e_m^b=1$ and $e_n=e_m^l$ for some non-zero integer $l$. In this case, ${\rm ord}(f)$ divides $b$. 
\item[$(ii)$]$e_m=-1$ and $e_n^{2a}=1$.
In this case, ${\rm ord}(f)$ divides $2a$. 
\end{enumerate}
Moreover, if ${\rm ord}(f)$ divides $lm$ where $l\geq2$ and $m:={\rm Max}\{a,b\}$,
then ${\rm ord}(f)=2a$, ${\rm Fix}(f)=\{Q_3,Q_4\}$, and $C_{a,b}/\langle f\rangle \cong\mathbb P^1$.
\end{pro}
\begin{proof}
Let $\ff$ be the defining equation of $C_{a,b}$.
By Lemma \ref{1}, the polynomial $F_{Q_1}$ contains the monomial $X_0^a Y_0^b$, and $F_{Q_2}$ contains $X_0^a Y_1^b$, whereas $F_{Q_3}$ contains $X_1^a Y_0^b$, and $F_{Q_4}$ does not contain the monomial $X_1^a Y_1^b$. Therefore, let
\begin{equation*}\label{lem2.1}
\begin{split}
F:=&s_{a,b}X_0^aY_0^b+s_{a,0}X_0^aY_1^b+s_{1,b}X_0X_1^{a-1}Y_0^b+s_{0,b-1}X_1^aY_0^{b-1}Y_1\\
	&+s_{1,0}X_0X_1^{a-1}Y_1^b+s_{0,1}X_1^aY_0Y_1^{b-1}
\end{split}
\end{equation*}
where $s_{a,b},s_{a,0}\in\mathbb C^{\ast}$ and $s_{1,b},s_{0,b-1},s_{1,0},s_{0,1}\in\mathbb C$ with
$(s_{1,b},s_{0,b-1})\not=(0,0)$ and $(s_{1,0},s_{0,1})\not=(0,0)$.
Note that $\mathbb{E} \setminus\{(a,b),\,(a,0),\,(0,b),\,(1,b),\,(0,b-1),\,(0,0),$
$ (1,0),\,(0,1)\}=\{(a-1,b),\,(a,b-1),\,(a-1,0),\,(a,1)\}$.
By the equation~\eqref{01}, 
\[
\sum_{i=1}^4F_{Q_i} = F + \sum_{(i,j) \in \{(a-1,b),\,(a,b-1),\,(a-1,0),\,(a,1)\}} s_{i,j} X_0^i X_1^{a-i} Y_0^j Y_1^{b-j},
\]
where $s_{i,j} \in \mathbb{C}$ for $(i,j) \in \{(a-1,b),\,(a,b-1),\,(a-1,0),\,(a,1)\}$.
Since $f$ is an automorphism of $C_{a,b}$, we have
$f^{\ast}\ff=t\ff$ for some $t\in\mathbb C^{\ast}$.
Thus, by the equation~\eqref{02}, $f^{\ast}F=tF$, i.e.
\begin{dmath*}
e_n^ae_m^bs_{a,b}X_0^aY_0^b+e_n^as_{a,0}X_0^aY_1^b+e_ne_m^bs_{1,b}X_0X_1^{a-1}Y_0^b+e_m^{b-1}s_{0,b-1}X_1^aY_0^{b-1}Y_1 +e_ns_{1,0}X_0X_1^{a-1}Y_1^b+e_ms_{0,1}X_1^aY_0Y_1^{b-1}\\
=t(s_{a,b}X_0^aY_0^b+s_{a,0}X_0^aY_1^b+s_{1,b}X_0X_1^{a-1}Y_0^b+s_{0,b-1}X_1^aY_0^{b-1}Y_1+s_{1,b}X_0X_1^{a-1}Y_0^b+s_{0,b-1}X_1^aY_0^{b-1}Y_1+s_{1,0}X_0X_1^{a-1}Y_1^b+s_{0,1}X_1^aY_0Y_1^{b-1}).
\end{dmath*}
By comparing the coefficients of the monomials $X_0^a Y_0^b$, $X_0^a Y_1^b$, and $X_0X_1^{a-1} Y_0^b$ and those of $X_1^a Y_0^{b-1}Y_1$, $X_0X_1^{a-1}Y_1^b$, and $X_1^a Y_0 Y_1^{b-1}$, and using the assumptions that $s_{a,b},s_{a,0}\in\mathbb C^{\ast}$, $(s_{1,b},s_{0,b-1})\not=(0,0)$, and $(s_{1,0},s_{0,1})\not=(0,0)$, we obtain the following equations:
\[e_n^ae_m^b=t,\ \ \ \ e_n^a=t,\ \ \ \ (e_ne_m^b-t)(e_m^{b-1}-t)=0,\ \ \ \ (e_n-t)(e_m-t)=0.\]
By the first and second equations, we obtain $e_m^b=1$.
The above equations are equal to the following equations:
\begin{equation}\label{eq2}
e_n^a=t,\ \ \ \ e_m^b=1,\ \ \  (e_n-t)(e_m^{b-1}-t)=0,\ \ \ \ (e_n-t)(e_m-t)=0.
\end{equation}
By the forth equation in $(\ref{eq2})$, either $e_n=t$ or $e_m=t$.

\textbf{Case 1:} We assume that $e_n=e_m=t$. 
By the first and second equations in $(\ref{eq2})$, we obtain $e_n^{a-1}=e_n^b=1$.
Thus, ${\rm ord}(f)$ divides ${\rm gcd}(a-1,b)$, and hence ${\rm ord}(f)$ divides $b$.

\textbf{Case 2:} We assume that $e_n=t$ and $e_m\not=t$.
By the equation $e_ms_{0,1}X_1^aY_0Y_1^{b-1}=ts_{0,1}X_1^aY_0Y_1^{b-1}$, we obtain $s_{0,1}=0$.
By the first equation in $(\ref{eq2})$, we have $e_n^{a-1}=1$.
From the third equation in $(\ref{eq2})$, we consider two cases separately: $e_m^{b-1}=e_n$ and $e_m^{b-1}\neq e_n$.
If $e_m^{b-1}=e_n$, then by the second equation in $(\ref{eq2})$, i.e. $e_m^b=1$, we get ${\rm ord}(f)$ divides $b$.
If $e_m^{b-1}\not=e_n$,
then by the equation $e_m^{b-1}s_{0,b-1}X_1^aY_0^{b-1}Y_1=ts_{0,b-1}X_1^aY_0^{b-1}Y_1$, we obtain $s_{0,b-1}=0$.
Since $s_{0,1}=s_{0,b-1}=0$, we get that $\sum_{i=1}F_{Q_i}$ is a multiple of $X_0$.
Since $C_{a,b}$ is smooth, 
it is in particular integral.
Thus, the defining equation $\ff$ of $C_{a,b}$ is irreducible.
Therefore, $\ff$ contains a term of the form $X_1^aY_0^lY_1^{b-l}$ for some $l$. 
Since $t=e_n$, $e_n=e_m^l$.
Thus, ${\rm ord}(f)$ divides $b$.

\textbf{Case 3:} We assume that $e_n\not=t$ and $e_m=t$.
By the third equation in $(\ref{eq2})$, we have $e_m^{b-2}=1$.
By the second equation in $(\ref{eq2})$, we obtain $e_m^b=1$.
Since $e_m^{b-2}=e_m^b=1$, it follows that $e_m^2=1$.
Since $m\geq 2$, we conclude that $m=2$, i.e. $e_m=-1$.
By the first equation in $(\ref{eq2})$, we have $e_n^a=-1$.
Thus, ${\rm ord}(f)$ divides $2a$.
From the above cases, parts $(i)$ and $(ii)$ of this theorem are shown.
		
Finally, we assume that ${\rm ord}(f)$ divides $kl$ where $k\geq2$ and $l:={\rm Max}\{a,b\}$.
By parts $(i)$ and $(ii)$ of this theorem,
we obtain ${\rm ord}(f)=2a$ and $f=[D(e_{2a},1)]\times[D(e_2,1)]$.
Then ${\rm Fix}(f)=\{Q_3,Q_4\}$.
Since $f^2=[D(e_{a},1)]\times[I_2]$, by Lemma \ref{th}, we have $C_{a,b}/\langle f\rangle \cong\mathbb P^1$.
\end{proof}
\begin{pro}\label{10}
Let $C_{a,b}\subset \pp$ be a smooth curve of bidegree $(a,b)$ with an automorphism $f$ such that $f=[D(e_n,1)]\times [D(e_m,1)]$
where $n,m\geq2$.
We assume that $\{Q_i\}_{i=1}^4 \cap C_{a,b}=\{Q_2,Q_3\}$.
Then one of the following holds:
\begin{enumerate}
\item[$(i)$]Either $e_n=e_m$ or $e_n=e_m^{-1}$, and either $e_n^{a-1}=1$ or $e_n^{b-1}=1$.
In this case, ${\rm ord}(f)$ divides $a-1$ or $b-1$.
\item[$(ii)$]$e_n\not=e_m$, $e_n\not=e_m^{-1}$, $e_n=e_m^{-b}$, and $e_m^{(a-1)b}=1$. 
In this case, ${\rm ord}(f)$ divides $(a-1)b$, and
\begin{equation*}
	\begin{split}
\sum_{i=1}^4F_{Q_i}=&s_{a,b}X_0^aY_0^b+s_{a-1,0}X_0^{a-1}X_1Y_1^b\\
&+s_{1,b}X_0X_1^{a-1}Y_0^b+s_{0,0}X_1^aY_1^b
\end{split}
\end{equation*}
where $s_{a,b},s_{a-1,0},s_{1,b},s_{0,0}\in\mathbb C^{\ast}$.
\item[$(iii)$]$e_n\not=e_m$, $e_n\not=e_m$, $e_m=e_n^{-a}$, and $e_n^{a(b-1)}=1$.
In this case, ${\rm ord}(f)$ divides $a(b-1)$, and
\begin{equation*}
\begin{split}
\sum_{i=1}^4F_{Q_i}=&s_{a,b}X_0^aY_0^b+s_{0,b-1}X_1^aY_0^{b-1}Y_1\\
&+s_{a,1}X_0^aY_0Y_1^{b-1}+s_{0,0}X_1^aY_1^b
\end{split}
\end{equation*}
where $s_{a,b},s_{0,b-1},s_{a,1},s_{0,0}\in\mathbb C^{\ast}$.
\end{enumerate}
Moreover, if ${\rm ord}(f)$ divides $kl$ where $k\geq2$ and $l:={\rm Max}\{a,b\}$,
then ${\rm Fix}(f)=\{Q_2,Q_3\}$ and $C_{a,b}/\langle f\rangle \cong\mathbb P^1$.
\end{pro}
\begin{proof}
Let $\ff$ be the defining equation of $C_{a,b}$.
By Lemma \ref{1}, the polynomial $F_{Q_1}$ contains the monomial $X_0^a Y_0^b$, and
$F_{Q_4}$ does contains the monomial $X_1^a Y_1^b$, whereas 
$F_{Q_2}$ does not contain $X_0^a Y_1^b$, and $F_{Q_3}$ does not contain $X_1^a Y_0^b$.
Therefore, let
\begin{equation*}\label{lem2.2}
\begin{split}
F:=&s_{a,b}X_0^aY_0^b+s_{a-1,0}X_0^{a-1}X_1Y_1^b+s_{a,1}X_0^aY_0Y_1^{b-1}\\
&+s_{1,b}X_0X_1^{a-1}Y_0^b+s_{0,b-1}X_1^aY_0^{b-1}Y_1+s_{0,0}X_1^aY_1^b
\end{split}
\end{equation*}
where $s_{a,b},s_{0,0}\in\mathbb C^{\ast}$ and $s_{a-1,0},s_{a,1},s_{1,b},s_{0,b-1}\in\mathbb C$ with
$(s_{a-1,0},s_{a,1})\not=(0,0)$ and $(s_{1,b},s_{0,b-1})\not=(0,0)$.
Note that $\mathbb{E} \setminus\{(a,b),\,(a,0),\,(a-1,0)\,(a,1),\,(0,b),\,(1,b),\,(0,b-1),\,(0,0)\}=\{(a-1,b),\,(a,b-1),\,(1,0),\,(0,1)\}$.
By the equation~\eqref{01}, 
\[
\sum_{i=1}^4F_{Q_i}  = F + \sum_{(i,j) \in \{(a-1,b),\,(a,b-1),\,(1,0),\,(0,1)\}} s_{i,j} X_0^i X_1^{a-i} Y_0^j Y_1^{b-j},
\]
where $s_{i,j} \in \mathbb{C}$ for $(i,j) \in \{(a-1,b),\,(a,b-1),\,(1,0),\,(0,1)\}$.
Since $f$ is an automorphism of $C_{a,b}$, it follows that
$f^{\ast}\ff=t\ff$ for some $t\in\mathbb C^{\ast}$.
By the equation~\eqref{02}, $f^{\ast}F=tF$, i.e.
\begin{dmath*}
e_n^ae_m^bs_{a,b}X_0^aY_0^b+e_n^{a-1}s_{a-1,0}X_0^{a-1}X_1Y_1^b+e^a_ne_ms_{a,1}X_0^aY_0Y_1^{b-1}+e_ne_m^bs_{1,b}X_0X_1^{a-1}Y_0^b+e_m^{b-1}s_{0,b-1}X_1^aY_0^{b-1}Y_1+s_{0,0}X_1^aY_1^b\\
=t(s_{a,b}X_0^aY_0^b+s_{a-1,0}X_0^{a-1}X_1Y_1^b+s_{a,1}X_0^aY_0Y_1^{b-1}+s_{1,b}X_0X_1^{a-1}Y_0^b+s_{0,b-1}X_1^aY_0^{b-1}Y_1+s_{0,0}X_1^aY_1^b)
\end{dmath*}
By comparing the coefficients of the monomials $X_0^aY_0^b$, $X_0^{a-1}X_1Y_1^b$, $X_0^aY_0Y_1^{b-1}$, $X_0X_1^{a-1}Y_0^b$, $X_1^aY_0^{b-1}Y_1$, and $X_1^aY_1^b$, and using the assumptions that $s_{a,b},s_{0,0}\in\mathbb C^{\ast}$,
$(s_{a-1,0},s_{a,1})\not=(0,0)$, and $(s_{1,b},s_{0,b-1})\not=(0,0)$,
we have $t=1$ and the following equations:
\begin{equation}\label{eq2.5}
e_n^ae_m^b=1,\ \ \ \ (e_n^{a-1}-1)(e_n^ae_m-1)=0,\ \ \ \ (e_ne_m^b-1)(e_m^{b-1}-1)=0.
\end{equation}	

\textbf{Case 1:}　We assume that $e_n=e_m$. 
The equations in $(\ref{eq2.5})$ are equal to the following equations:
\[	e_n^{a+b}=1,\ \ \ \ \ (e_n^{a-1}-1)(e_n^{a+1}-1)=0,\ \ \ \ \ 
(e_n^{b+1}-1)(e_n^{b-1}-1)=0.\]
By \ta second equation, $e_n^{a-1}=1$ or $e_n^{a+1}=1$.
If $e_n^{a-1}=1$, then ${\rm ord}(f)$ divides $a-1$.
If $e_n^{a+1}=1$, then by \ta first equation, we get that $e_n^{b-1}=1$.
Since $e_n^{a+1}=e_n^{b-1}=1$, ${\rm ord}(f)$ divides ${\rm gcd}(a+1,b-1)$. 
Thus, ${\rm ord}(f)$ divides $b-1$.			
			
Now, we assume that $e_n\not=e_m$.
By the second equation in $(\ref{eq2.5})$,
either $e_n^{a-1}=1$ or $e_n^ae_m=1$.
By the third equation in $(\ref{eq2.5})$,
either $e_ne_m^b=1$ or $e_m^{b-1}=1$.

\textbf{Case 2:}　We assume that $e_n\not=e_m$, and $e_n^{a-1}=e_n^ae_m=1$.
In this case, we have $e_ne_m=1$, i.e. $e_n=e_m^{-1}$.
Since $e_n^{a-1}=1$, we obtain ${\rm ord}(f)$ divides $a-1$.

By the symmetry of $a,n$ and $b,m$ in the equations $(\ref{eq2.5})$,
if $e_ne_m^b=e_m^{b-1}=1$, then $e_n=e_m^{-1}$ and ${\rm ord}(f)$ divides $a-1$.

Therefore, we assume that $e_n\not=e_m$, only one of $e_n^{a-1}=1$ or $e_n^ae_m=1$ holds, and only one of $e_ne_m^b=1$ or $e_m^{b-1}=1$ holds.

\textbf{Case 3:}
We assume that $e_n\not=e_m$,
$e_n^{a-1}=1$, and $e_n^ae_m\not=1$.
By the equation $e_n^ae_ms_{a,1}X_0^aY_0Y_1^{b-1}=s_{a,1}X_0^aY_0Y_1^{b-1}$,
we obtain $s_{a,1}=0$.
By the first equation in $(\ref{eq2.5})$, 
we get
$e_ne_m^b=1$, i.e. $e_n=e_m^{-b}$.
Since $e_n^{a-1}=1$, it follows that $e_m^{(a-1)b}=1$.
Thus, ${\rm ord}(f)$ divides $(a-1)b$.
Additionally,
by the third equation in $(\ref{eq2.5})$ and $e_ne_m^b=1$,
we have $e_m^{b-1}\not=1$.
By the equation $e_m^{b-1}s_{0,b-1}X_1^aY_0^{b-1}Y_1=s_{0,b-1}X_1^aY_0^{b-1}Y_1$,
we obtain $s_{0,b-1}=0$.
Since $t=1$, we have $f^*\left(\sum_{i=1}^4F_{Q_i}\right)=\sum_{i=1}^4F_{Q_i}$.
Then $(e_n^ie_m^j-1)s_{i,j}=0$ for $(i,j) \in \{(a-1,b),\,(a,b-1),\,(1,0),\,(0,1)\}$.
Since $e_n^{a-1}=1$, $e_n=e_m^{-b}$, and $n,m\geq2$, we get that $s_{i,j}=0$ for $(i,j) \in \{(a-1,b),\,(a,b-1),\,(1,0),\,(0,1)\}$.
Thus, $\sum_{i=1}^4F_{Q_i}=s_{a,b}X_0^aY_0^b+s_{a-1,0}X_0^{a-1}X_1Y_1^b+s_{1,b}X_0X_1^{a-1}Y_0^b+s_{0,0}X_1^aY_1^b$.
In addition,
if $e_n=e_m^{-1}$, then by $e_ne_m^b=1$, we obtain $e_m^{-1}=e_n$.
This contradicts that $e_m^{-1}\not=e_n$.

\textbf{Case 4:}
We assume that $e_n\not=e_m$,
$e_n^{a-1}\not=1$, and $e_n^ae_m=1$.
By the equation $e_n^{a-1}s_{a-1,0}X_0^{a-1}X_1Y_0^b=s_{a-1,0}X_0^{a-1}X_1Y_0^b$,
we obtain $s_{a-1,0}=0$.
By the first equation in $(\ref{eq2.5})$, 
		$e_m^{b-1}=1$.
Since $e_n^ae_m=1$, i.e. $e_m=e_n^{-a}$ and $e_m^{b-1}=1$,
we have $e_n^{a(b-1)}=1$.
Thus, ${\rm ord}(f)$ divides $a(b-1)$.
Additionally,
	by the third equation in $(\ref{eq2.5})$ and  $e_m^{b-1}=1$,
we obtain $e_ne_m^b\not=1$.
By the equation $e_ne_m^bs_{1,b}X_0X_1^{a-1}Y_0^b=s_{1,b}X_0X_1^{a-1}Y_0^b$,
we obtain $s_{1,b}=0$.
Since $t=1$, we have $f^*\left(\sum_{i=1}^4F_{Q_i}\right)=\sum_{i=1}^4F_{Q_i}$.
Then $(e_n^ie_m^j-1)s_{i,j}=0$ for $(i,j) \in \{(a-1,b),\,(a,b-1),\,(1,0),\,(0,1)\}$.
Since $e_m^{b-1}=1$, $e_n^ae_m=1$, and $n,m\geq2$, we get that $s_{i,j}=0$ for $(i,j) \in \{(a-1,b),\,(a,b-1),\,(1,0),\,(0,1)\}$.
Thus, $\sum_{i=1}^4F_{Q_i}=s_{a,b}X_0^aY_0^b+s_{0,1}X_1^aY_0^{b-1}Y_1+s_{a,1}X_0^aY_0Y_1^{b-1}+s_{0,0}X_1^aY_1^b$.
In addition,
If $e_n=e_m^{-1}$, then by $e_n^ae_m=1$, we obtain $e_n^{a-1}=1$.
This contradicts that $e_n^{a-1}\not=1$.
From the above cases, parts $(i)$, $(ii)$, and $(iii)$ of this theorem are shown.

Finally, 
we assume that ${\rm ord}(f)$ divides $kl$ where $k\geq2$ and $l:={\rm Max}\{a,b\}$.
When $a={\rm Max}\{a,b\}$, $f$ corresponds to part $(iii)$ of this theorem, i.e. $f=[D(e_{ka},1)]\times [D(e_{ka}^{-a},1)]$.
Then $f^k=[D(e_{ak}^k,1)]\times[I_2]$.
Since $e_{ak}^k$ is a primitive $a$-th root of unity, and by Lemma \ref{th}, it follows that $C_{a,b}/\langle f\rangle \cong\mathbb P^1$.
Similarly, if $b={\rm Max}\{a,b\}$,
then $f$ corresponds to part $(ii)$ of this theorem, i.e. $f=[D(e_{kb}^{-b},1)]\times [D(e_{kb},1)]$.
Then $f^k=[I_2]\times[D(e_{kb}^k,1)]$.
Since $e_{bk}^k$ is a primitive $b$-th root of unity, and by Lemma \ref{th}, we have $C_{a,b}/\langle f\rangle \cong\mathbb P^1$.
\end{proof}
\begin{thm}\label{11}
Let $C_{a,b}\subset \pp$ be a smooth curve of bidegree $(a,b)$ with an automorphism $f$.
We assume that $|\{Q_i\}_{i=1}^4\cap C_{a,b}|=2$.
\begin{enumerate}
\item[$(i)$]If $f=[D(e_n,1)]\times [D(e_m,1)]$ where $n,m\geq 2$, then $|{\rm Fix}(f)|>0$, and ${\rm ord}(f)$ divides either $2a$, $2b$, $(a-1)b$, or $a(b-1)$.
\item[$(ii)$]If $f=s_{(1,2)}\circ ([A]\times [B])$ where $A$ and $B$ are diagonal matrices such that $AB=BA=D(e_n,1)$ with $n\geq2$,
then $a=b$, and ${\rm ord}(f)$ divides either $4$ or $2(a-1)$.
\item[$(iii)$]If ${\rm ord}(f)$ divides $lm$ where $l\geq2$ and $m:={\rm Max}\{a,b\}$,
then $|{\rm Fix}(f)|>0$ and $C_{a,b}/\langle f\rangle \cong\mathbb P^1$.
\end{enumerate}
\end{thm}
\begin{proof}
First, we assume that $f=[D(e_n,1)]\times [D(e_m,1)]$ where $n,m\geq2$.
By Lemma \ref{7},
there are an automorphism $h\in{\rm Aut}(\pp)$, a smooth curve 
$C'_{s,t}\subset \pp$ of bidegree $(s,t)$, and an automorphism $g=[D(e_k,1)]\times [D(e_l,1)]\in{\rm Aut}(C'_{s,t})$ 
 where
$\{s,t\}=\{a,b\}$ and $\{n,m\}=\{k,l\}$
such that 
$\{Q_i\}_{i=1}^4\cap C'_{s,t}=\{Q_3,Q_4\}$ or $\{Q_2,Q_3\}$,
$h(C'_{s,t})=C_{a,b}$, and $h^{-1}\circ f\circ h=g$.
Then ${\rm ord}(f)={\rm ord}(g)$, $|{\rm Fix}(f)|=|{\rm Fix}(g)|$, and $C_{a,b}/\langle f \rangle\cong C'_{s,t}/\langle g\rangle$.

If $\{Q_i\}_{i=1}^4\cap C'_{s,t}=\{Q_3,Q_4\}$, then
by Proposition \ref{8.1},
${\rm ord}(g)$ divides $t$ or $2s$.
Since $\{s,t\}=\{a,b\}$ and ${\rm ord}(f)={\rm ord}(g)$,
we obtain ${\rm ord}(f)$ divides either $2a$ or $2b$.
Since ${\rm Fix}(g)=\{Q_3,Q_4\}$, we have $|{\rm Fix}(f)|=2$.
Additionally, 
we assume that ${\rm ord}(f)$ divides $lm$ where $l\geq2$ and $m:={\rm Max}\{a,b\}$.
By Proposition \ref{8.1}, $C'_{s,t}/\langle g\rangle \cong \mathbb P^1$, and hence $C_{a,b}/\langle f\rangle \cong \mathbb P^1$.

If $\{Q_i\}_{i=1}^4\cap C'_{s,t}=\{Q_2,Q_3\}$, then
by Proposition \ref{10} and $\{s,t\}=\{a,b\}$, we obtain
${\rm ord}(g)$ divides either $a-1$, $b-1$ $(a-1)b$ or $a(b-1)$.
Thus, ${\rm ord}(f)$ divides either $(a-1)b$ or $a(b-1)$.
Since ${\rm Fix}(g)=\{Q_2,Q_3\}$, we have $|{\rm Fix}(f)|=2$.
Additionally, 
we assume that ${\rm ord}(f)$ divides $lm$ where $l\geq2$ and $m:={\rm Max}\{a,b\}$.
By Proposition \ref{10}, $C'_{s,t}/\langle g\rangle \cong \mathbb P^1$, and hence $C_{a,b}/\langle f\rangle \cong \mathbb P^1$.
	
Next, we assume that $f=s_{(1,2)}\circ ([A]\times [B])$ where $A$ and $B$ are diagonal matrices such that $AB=BA=D(e_n,1)$ with $n\geq2$.
By Lemma \ref{8}, $a=b$.
By Lemma \ref{7},
$\{Q_i\}_{i=1}^4 \cap C_{a,b}$ is $\{Q_2,Q_3\}$ or $\{Q_1,Q_4\}$.

We assume that $\{Q_i\}_{i=1}^4 \cap C_{a,b}=\{Q_2,Q_3\}$.
Since $f^2=[D(e_n,1)]\times[D(e_n,1)]$, $f^2$ corresponds to part $(i)$ of Proposition \ref{10}.
Then ${\rm ord}(f^2)$ divides $a-1$, and hence ${\rm ord}(f)$ divides $2(a-1)$.

We assume that $\{Q_i\}_{i=1}^4 \cap C_{a,b}=\{Q_1,Q_4\}$.
By exchanging $X_0$ and $X_1$ in the first component of 
$\pp$, we obtain 
$\{Q_i\}_{i=1}^4 \cap C_{a,b}=\{Q_2,Q_3\})$ and $f^2=[D(e_n^{-1},1)]\times [D(e_n,1)]$.
Then $f^2$ corresponds to the cases $(i)$ of Proposition \ref{10}.
Then ${\rm ord}(f)$ divides $2(a-1)$.
Note that $2(a-1)<lm$ where $l\geq 2$ and $m={\rm Max}\{a,b\}$.
\end{proof}
%
%
\begin{cro}\label{-2}
Let $C_{a,b}\subset \pp$ be a smooth curve of bidegree $(a,b)$ with an automorphism $f$.
We assume that $|\{Q_i\}_{i=1}^4\cap C_{a,b}|=2$.
If ${\rm ord}(f)>2m$ where $m:={\rm Max}\{a,b\}$,
then by replacing the coordinate systems of the first and second components of $\pp$ individually if necessary,
$\{Q_i\}_{i=1}^4\cap C_{a,b}=\{Q_2,Q_3\}$.
\end{cro}
\begin{proof}
By Lemmas \ref{2}, \ref{-1}, and \ref{8} and Theorem \ref{11},
we may assume that $f=[D(e_n,1)]\times [D(e_m,1)]$ where $n,m\geq 2$.
By Lemma \ref{7} and its proof, 
if $\{Q_i\}_{i=1}^4\cap C_{a,b}$ is either $\{Q_1,Q_2\}$, $\{Q_2,Q_4\}$, or $\{Q_1,Q_3\}$,
then by replacing the coordinate system if necessary,
we may assume that $f$ is an automorphism of smooth curve $C_{s,t}$ of bidegree $(s,t)$ sch that $\{Q_i\}_{i=1}^4\cap C_{a,b}=\{Q_3,Q_4\}$ and $\{a,b\}=\{s,t\}$.
By Proposition \ref{8.1}, ${\rm ord}(f)<2m$.
This is a contradiction.
Thus, $\{Q_i\}_{i=1}^4\cap C_{a,b}$ is either $\{Q_2,Q_3\}$ or $\{Q_1,Q_4\}$.
If  $\{Q_i\}_{i=1}^4\cap C_{a,b}=\{Q_1,Q_4\}$, then as like the proof of Lemma \ref{7},
by exchanging $X_0$ and $X_1$ in the first component of $\pp$,
we have $\{Q_i\}_{i=1}^4\cap C_{a,b}=\{Q_2,Q_3\}$.
\end{proof}

Next, we consider the case where $|\{Q_i\}_{i=1}^4 \cap C_{a,b}|=3$.
\begin{lem}\label{13--}
Let $C\subset \pp$ be a smooth curve with an automorphism $f$ such that $|\{Q_i\}_{i=1}^4 \cap C|=3$.
\begin{enumerate}
\item[$(i)$]We assume that $f=[D(e_n,1)]\times [D(e_m,1)]$. 
By replacing the coordinate systems of the first and second components of $\pp$ individually if necessary,
we get that $\{Q_i\}_{i=1}^4 \cap C=\{Q_2,Q_3,Q_4\}$, and $f=[D(e_s^{\epsilon},1)]\times [D(e_t^{\delta},1)]$ where $\{s,t\}=\{n,m\}$ and $\epsilon,\delta\in\{{\pm 1}\}$. 
\item[$(ii)$]We assume that $f=s_{(1,2)}\circ ([A]\times [B])$ such that $AB=BA=D(e_n,1)$, and $A$ and $B$ are diagonal matrices. 
Then $\{Q_i\}_{i=1}^4 \cap C$ is $\{Q_1,Q_2,Q_3\}$ or $\{Q_2,Q_3,Q_4\}$.
\end{enumerate}
\end{lem}
\begin{proof}
We assume that $f=[D(e_n,1)]\times [D(e_m,1)]$ where $n,m\geq 2$.
If $\{Q_i\}_{i=1}^4 \cap C=\{Q_1,Q_2,Q_4\}$ $($resp. $\{Q_1,Q_3,Q_4\})$, then by exchanging $Y_0$ and $Y_1$ in the second component (resp. by exchanging $X_0$ and $X_1$ in the first component) of 
	$\pp$, we obtain 
	$\{Q_i\}_{i=1}^4 \cap C=\{Q_2,Q_3,Q_4\}$.
In these cases,	$f$ is $[D(e_n,1)]\times[D(e_m^{-1},1)]$ (resp. $f$ is $[D(e_n^{-1},1)]\times[D(e_m,1)]$).

If $\{Q_i\}_{i=1}^4 \cap C=\{Q_1,Q_2,Q_3\}$, then by exchanging $X_0$ and $X_1$ in the first component and 
exchanging $Y_0$ and $Y_1$ in the second component of $\pp$,
we have 
$\{Q_i\}_{i=1}^4 \cap C=\{Q_2,Q_3,Q_4\}$.
In this case, $f$ is $[D(e_n^{-1},1)]\times[D(e_m^{-1},1)]$.
	
Next, we assume that $f=\s\circ ([A]\times [B])$ such that $AB=BA=D(e_n,1)$, and $A$ and $B$ are diagonal matrices. 
Since $A$ and $B$ are diagonal matrices,
$[A]\times[B](Q_i)=Q_i$ for $i=1,\ldots, 4$.
Since $\s(Q_1)=Q_1$, $\s(Q_4)=Q_4$, $\s(Q_2)=Q_3$, and $|\{Q_i\}_{i=1}^4 \cap C|=3$, 
we get that $\{Q_i\}_{i=1}^4 \cap C$ is either $\{Q_1,Q_2,Q_3\}$ or $\{Q_2,Q_3,Q_4\}$.
\end{proof}
\begin{pro}\label{13-}
Let $C_{a,b}\subset \pp$ be a smooth curve of bidegree $(a,b)$ with an automorphism $f=[D(e_n,1)]\times [D(e_m,1)]$ where $n,m\geq 2$.
We assume that $|\{Q_i\}_{i=1}^4 \cap C_{a,b}|=\{Q_2,Q_3,Q_4\}$.
Then ${\rm ord}(f)$ divides
${\rm gcd}(a-2,2b-1)$ or ${\rm gcd}(2a-1,b-2)$.
\end{pro}
\begin{proof}
Let $\ff$ be the defining equation of $C_{a,b}$.
By Lemma \ref{1}, the polynomial $F_{Q_1}$ contains the monomial $X_0^a Y_0^b$, whereas 
$F_{Q_2}$ does not contain $X_0^a Y_1^b$, 
$F_{Q_3}$ does not contain $X_1^a Y_0^b$, and
$F_{Q_4}$ does not contain the monomial $X_1^a Y_1^b$. Therefore, let
\begin{equation}\label{lem3}
\begin{split}
F:=&s_{a,b}X_0^aY_0^b+s_{a-1,0}X_0^{a-1}X_1Y_1^b+s_{a,1}X_0^aY_0Y_1^{b-1}+s_{1,b}X_0X_1^{a-1}Y_0^b\\
		&+s_{0,b-1}X_1^aY_0^{b-1}Y_1+s_{1,0}X_0X_1^{a-1}Y_1^b+s_{0,1}X_1^aY_0Y_1^{b-1}
	\end{split}
\end{equation}
where $s_{a,b}\in\mathbb C^{\ast}$ and $s_{a-1,0},s_{a,1},s_{1,b},s_{0,b-1},s_{1,0},$
$s_{0,1}\in\mathbb C$ with
$(s_{a-1,0},$
$s_{a,1})\not=(0,0)$, $(s_{1,b},s_{0,b-1})\not=(0,0)$, and $(s_{1,0},s_{0,1})\not=(0,0)$.
Note that $\mathbb{E} \setminus\{(a,b), (a,0),(a-1,0), (a,1),(0,b), (1,b), (0,b-1),(0,0), (1,0), (0,1)\}=\{(a-1,b),\,(a,b-1)\}$.
Then, by the equation~\eqref{01}, 
\[
\sum_{i=1}^4F_{Q_i} = F + s_{a-1,b} X_0^{a-1} X_1 Y_0^b+s_{a,b-1} X_0^aY_0^{b-1}Y_1
\]
where $s_{a-1,b},s_{a,b-1} \in \mathbb{C}$.
Since $f$ is an automorphism of $C_{a,b}$, it follows that
$f^{\ast}\ff=t\ff$ for some $t\in\mathbb C^{\ast}$.
Thus, by the equation~\eqref{02}, $f^{\ast}F=tF$, i.e.
	\begin{dmath*}
		e_n^ae_m^bs_{a,b}X_0^aY_0^b+e_n^{a-1}s_{a-1,0}X_0^{a-1}X_1Y_1^b+e_n^ae_ms_{a,1}X_0^aY_0Y_1^{b-1}+e_ne_m^bs_{1,b}X_0X_1^{a-1}Y_0^b+e_m^{b-1}s_{0,b-1}X_1^aY_0^{b-1}Y_1+e_ns_{1,0}X_0X_1^{a-1}Y_1^b+e_ms_{0,1}X_1^aY_0Y_1^{b-1}\\
		=t(s_{a,b}X_0^aY_0^b+s_{a-1,0}X_0^{a-1}X_1Y_1^b+s_{a,1}X_0^aY_0Y_1^{b-1}+s_{1,b}X_0X_1^{a-1}Y_0^b+s_{0,b-1}X_1^aY_0^{b-1}Y_1+s_{1,0}X_0X_1^{a-1}Y_1^b+s_{0,1}X_1^aY_0Y_1^{b-1}).
	\end{dmath*}
By comparing the coefficients of the monomials $X_0^aY_0^b$, $X_0^{a-1}X_1Y_1^b$, $X_0^aY_0Y_1^{b-1}$, $X_0X_1^{a-1}Y_0^b$, $X_1^aY_0^{b-1}Y_1$, $X_0X_1^{a-1}Y_1^b$, and $X_1^aY_0Y_1^{b-1}$, and using the assumptions that $s_{a,b}\in\mathbb C^{\ast}$,
$(s_{a-1,0},$
$s_{a,1})\not=(0,0)$, $(s_{1,b},s_{0,b-1})\not=(0,0)$, and $(s_{1,0},s_{0,1})\not=(0,0)$, we obtain the following equations:
	\begin{equation}\label{eq3}
		\begin{aligned}
			e_n^ae_m^b&=t,\ \ \ \ \ \ \ \ &(e_n^{a-1}-t)(e_n^ae_m-t)&=0,\\
			(e_ne_m^b-t)(e_m^{b-1}-t)&=0,\ \ \ \ \ \ \ \ &(e_n-t)(e_m-t)&=0.
		\end{aligned}
	\end{equation}
By the forth equation in $(\ref{eq3})$, either $e_n=t$ or $e_m=t$.

\textbf{Case 1:} We assume that $e_n=e_m=t$. 
	By the equation $(e_n-t)(e_m-t)=0$ in $(\ref{eq3})$, $e_n=t$.
	The equations in $(\ref{eq3})$ are equal to the following equations:
	\[	e_n^{a+b-1}=1,\ \ \ \ \ \ (e_n^{a-2}-1)(e_n^a-1)=0,\ \ \ \ \ \ 
	(e_n^b-1)(e_n^{b-2}-1)=0.\]
	By \ta second equation, either $e_n^{a-2}=1$ or $e_n^a=1$.
	If $e_n^{a-2}=1$,
	then by the above first equation, we have $e_n^{b+1}=1$. 
Since $n\geq 2$,
the third equation implies $e_n^{b-2}=1$, and hence $e_n^3=1$. 
	Then $a-2$ and $b-2$ are multiples of $3$, and ${\rm ord}(f)=3$.
since $a-2$ and $b-2$ are multiples of $3$, we get ${\rm gcd}(a-2,2b-1)=3$.
Thus, ${\rm ord}(f)$ divides ${\rm gcd}(a-2,2b-1)$.
	If $e_n^a=1$,
	then by the above first equation, $e_n^{b-1}=1$. 
	Since $n\geq 2$, this contradicts the above third equation.
	
\textbf{Case 2:} We assume that $e_n\not=e_m$, and $e_n=t$.
The equations in $(\ref{eq3})$ are equal to the following equations:
	\[e_n^{a-1}e_m^b=1,\ \ \ \ \ \ (e_n^{a-2}-1)(e_n^{a-1}e_m-1)=0,\ \ \ \ \ \ (e_m^b-1)(e_m^{b-1}-e_n)=0.\]
	By the above third equation, either $e_m^b=1$ or $e_m^{b-1}=e_n$.
If $e_m^b=1$, then \ta first equation implies $e_n^{a-1}=1$.
By \ta second equation, it follows that $e_n=1$ or $e_m=1$.
This contradicts that $n,m\geq 2$.
Therefore, $e_m^b\not=1$ and $e_m^{b-1}=e_n$.
By \ta first equation, we obtain $e_n^ae_m=1$.
By the above second equation and $e_n\not=1$, 
we get 
$e_n^{a-2}=1$ and $e_n^{a-1}e_m\not=1$.
	Since $e_n^{a-1}e_m^b=1$, and $e_n^{a-2}=1$, 
we have $e_m=e_n^{-2}$.
Since $e_n^{a-2}=1$ and $e_m=e_n^{-2}$,
${\rm ord}(f)$ divides $a-2$.
Since $e_m^{b-1}=e_n$ and $e_m=e_n^{-2}$, we get that $e_m^{2b-1}=1$.
Since $e_m^{b-1}=e_n$, ${\rm ord}(f)$ divides $2b-1$.
Therefore, ${\rm ord}(f)$ divides ${\rm gcd}(a-2,2b-1)$.
	
\textbf{Case 3:} We assume that $e_n\not=e_m$, and $e_m=t$.
By the symmetry of $a,n$ and $b,m$ in the equations $(\ref{eq3})$,
we get that ${\rm ord}(f)$ divides ${\rm gcd}(2a-1,b-2)$. 
\end{proof}
\begin{thm}\label{13}
Let $C_{a,b}\subset \pp$ be a smooth curve of bidegree $(a,b)$ with an automorphism $f$.
We assume that $|\{Q_i\}_{i=1}^4 \cap C_{a,b}|=3$.
\begin{enumerate}
\item[$(i)$]If $f=[D(e_n,1)]\times [D(e_m,1)]$ where $n,m\geq2$, then ${\rm ord}(f)$ divides either $a-2$ or $b-2$.
\item[$(ii)$]If $f=s_{(1,2)}\circ ([A]\times [B])$ where $A$ and $B$ are diagonal matrices such that $AB=BA=D(e_n,1)$ with $n\geq2$,
then ${\rm ord}(f)=6$.
\item[$(iii)$]If $a\geq 3$, then ${\rm ord}(f)$ does not divide $lm$ where $l\geq2$ and $m:={\rm Max}\{a,b\}$.
\end{enumerate}
\end{thm}
\begin{proof}
First, we assume that $f=[D(e_n,1)]\times [D(e_m,1)]$ where $n,m\geq 2$.
By Lemma \ref{13--},
we may assume that $\{Q_i\}_{i=1}^4 \cap C_{a,b}=\{Q_2,Q_3,Q_4\}$.
By Proposition \ref{13-}, ${\rm ord}(f)$ divides ${\rm gcd}(a-2,2b-1)$ or ${\rm gcd}(2a-1,b-2)$.
Thus, ${\rm ord}(f)$ divides either $a-2$ or $b-2$.

Next, we assume that $f=s_{(1,2)}\circ ([A]\times [B])$ where $A$ and $B$ are diagonal matrices such that $AB=BA=D(e_n,1)$ with $n\geq2$.
By Lemma \ref{8}, $a=b$.
By Lemma \ref{13--}, is $\{Q_i\}_{i=1}^4 \cap C=\{Q_1,Q_2,Q_3\}$ or $\{Q_2,Q_3,Q_4\}$.

We assume that $\{Q_i\}_{i=1}^4 \cap C=\{Q_2,Q_3,Q_4\}$.
By Proposition \ref{13-},
${\rm ord}(f^2)$ divides ${\rm gcd}(a-2,2b-1)$ or ${\rm gcd}(2a-1,b-2)$.
Since $a=b$, ${\rm gcd}(a-2,2b-1)={\rm gcd}(2a-1,b-2)={\rm gcd}(a-2,3)$.
Since $n\geq 2$,
 we get that $a-2$ divides $3$, and ${\rm ord}(f^2)=3$. 
Thus, ${\rm ord}(f)=6$.

We assume that $\{Q_i\}_{i=1}^4 \cap C=\{Q_1,Q_2,Q_3\}$.
By exchanging $X_0$ and $X_1$ in the first component and $Y_0$ and $Y_1$ in the second component of 
$\pp$, we obtain $\{Q_i\}_{i=1}^4 \cap C=\{Q_2,Q_3,Q_4\}$.
Then $f^2=[D(e_n^{-1},1)]\times [D(e_n^{-1},1)]$.
Similarly to the case $\{Q_i\}_{i=1}^4 \cap C=\{Q_2,Q_3,Q_4\}$, we obtain that ${\rm ord}(f)=6$ and $a-2$ divides $3$.

From the above, parts $(i)$ and $(ii)$ of this theorem are shown.

Finally, we assume that $a\geq 3$ and ${\rm ord}(f)$ divides $lm$ where $l\geq2$ and $m:={\rm Max}\{a,b\}$.
Since ${\rm ord}(f)>a-2$ and ${\rm ord}(f)>b-2$,
$f$ corresponds to part $(ii)$ of this theorem.
Therefore,
we have
$a=b$, $a-2$ divides $3$, and ${\rm ord}(f)=6$.
Since $l\geq 2$, we get that $l=2$ and $a=3$.
This contradicts that $a-2$ divides $3$.
\end{proof}
%
%
%
We consider the case where $|\{Q_i\}_{i=1}^4 \cap C_{a,b}|=4$.
\begin{pro}\label{16-}
Let $C_{a,b}\subset \pp$ be a smooth curve of bidegree $(a,b)$
with an automorphism $f=[D(e_n,1)]\times [D(e_m,1)]$ where $n,m\geq 2$.
We assume that $|\{Q_i\}_{i=1}^4 \cap C_{a,b}|=4$.
Then one of the following holds:
\begin{enumerate}
\item[$(i)$]Either $e_n=e_m$ or $e_n=e_m^{-1}$.
In this case, ${\rm ord}(f)$ divides either ${\rm gcd}(a,b-2)$ or ${\rm gcd}(a-2,b)$.
\item[$(ii)$]$e_n\not=e_m$ and $e_m=e_n^i$ for some $2\leq i\leq b-2$, and $e_n^a=e_m^{b-2}=1$. In this case, ${\rm ord}(f)$ divides $a$.
\item[$(iii)$]$e_n\not=e_m$ and $e_n=e_m^i$ for some $2\leq i\leq a-2$, and $e_n^{a-2}=e_m^b=1$. In this case, ${\rm ord}(f)$ divides $b$.
\item[$(iv)$]$e_n\not=e_m$, $e_n\not=e_m^{-1}$, $e_m^{b-1}=e_n$, and $e_m^{(a-1)(b-1)+1}=1$.
In this case, ${\rm ord}(f)$ divides $(a-1)(b-1)+1$, and
\begin{dmath*}
\sum_{i=1}^4F_{Q_i}=s_{a-1,b}X_0^{a-1}X_1Y_0^b+s_{a,1}X_0^aY_0Y_1^{b-1}+s_{0,b-1}X_1^aY_0^{b-1}Y_1+
s_{1,b}X_0X_1^{a-1}Y_1^b
\end{dmath*}
where $s_{a-1,b},s_{a,1},s_{0,b-1},s_{1,b}\in\mathbb C^{\ast}$.
\item[$(v)$]$e_n\not=e_m$, $e_n\not=e_m^{-1}$, $e_n^{a-1}=e_m$, and $e_n^{(a-1)(b-1)+1}=1$.
In this case, ${\rm ord}(f)$ divides $(a-1)(b-1)+1$, and
\begin{dmath*}
\sum_{i=1}^4F_{Q_i}=s_{a,b-1}X_0^aY_0^{b-1}Y_1+s_{1,b}X_0X_1^{a-1}Y_0^b+s_{a-1,0}X_0^{a-1}X_1Y_1^b+s_{0,1}X_1^aY_0Y_1^{b-1}
\end{dmath*}
where $s_{a,b-1},s_{1,b},s_{a-1,0},s_{0,1}\in\mathbb C^{\ast}$.
\end{enumerate}	
\end{pro}
\begin{proof}
Let $\ff$ be the defining equation of $C_{a,b}$.
By Lemma \ref{1}, the polynomial $F_{Q_1}$ does not contain the monomial $X_0^a Y_0^b$, $F_{Q_2}$ does not contain $X_0^a Y_1^b$, $F_{Q_3}$ does not contain $X_1^a Y_0^b$, and $F_{Q_4}$ does not contain the monomial $X_1^a Y_1^b$. Therefore,
\begin{equation}\label{lem4}
\begin{split}
\sum_{i=1}^4F_{Q_i} =&s_{a-1,b}X_0^{a-1}X_1Y_0^b+s_{a,b-1}X_0^aY_0^{b-1}Y_1+s_{a-1,0}X_0^{a-1}X_1Y_1^b\\
		&+s_{a,1}X_0^aY_0Y_1^{b-1}+s_{1,b}X_0X_1^{a-1}Y_0^b+s_{0,b-1}X_1^aY_0^{b-1}Y_1\\
		&+s_{1,0}X_0X_1^{a-1}Y_1^b+s_{0,1}X_1^aY_0Y_1^{b-1}
\end{split}
\end{equation}
where $s_{a-1,b},s_{a,b-1},s_{a-1,0},s_{a,1},s_{1,b},s_{0,b-1},s_{1,0},s_{0,1}\in\mathbb C$ such that
$(s_{a-1,b},$
$s_{a,b-1})\not=(0,0)$, 
$(s_{a-1,0},s_{a,1})\not=(0,0)$, 
$(s_{1,b},s_{0,b-1})\not=(0,0)$, and $(s_{1,0},s_{0,1})$
$\not=(0,0)$.	
Since $f$ is an automorphism of $C_{a,b}$, we have
$f^{\ast}\ff=t\ff$ for some $t\in\mathbb C^{\ast}$.
Thus, by the equation~\eqref{02}, $f^{\ast}F_Q=tF_Q$, i.e.
\begin{dmath*}
			e_n^{a-1}e_m^bs_{a-1,b}X_0^{a-1}X_1Y_0^b+e_n^ae_m^{b-1 }s_{a,b-1}X_0^aY_0^{b-1}Y_1+e_n^{a-1}s_{a-1,0}X_0^{a-1}X_1Y_1^b+e_n^ae_ms_{a,1}X_0^aY_0Y_1^{b-1}+e_ne_m^bs_{1,b}X_0X_1^{a-1}Y_0^b+e_m^{b-1}s_{0,b-1}X_1^aY_0^{b-1}Y_1+e_ns_{1,0}X_0X_1^{a-1}Y_1^b+e_ms_{0,1}X_1^aY_0Y_1^{b-1}\\
			=t(s_{a-1,b}X_0^{a-1}X_1Y_0^b+s_{a,b-1}X_0^aY_0^{b-1}Y_1+s_{a-1,0}X_0^{a-1}X_1Y_1^b+s_{a,1}X_0^aY_0Y_1^{b-1}+s_{1,b}X_0X_1^{a-1}Y_0^b+s_{0,b-1}X_1^aY_0^{b-1}Y_1+s_{1,0}X_0X_1^{a-1}Y_1^b+s_{0,1}X_1^aY_0Y_1^{b-1}).
\end{dmath*}
By comparing the coefficients of the monomials $X_0^{a-1}X_1Y_0^b$, $X_0^aY_0^{b-1}Y_1$, $X_0^{a-1}X_1Y_1^b$, $X_0^aY_0Y_1^{b-1}$, $X_0X_1^{a-1}Y_0^b$, $X_1^aY_0^{b-1}Y_1$, $X_0X_1^{a-1}Y_1^b$, and $X_1^aY_0Y_1^{b-1}$, and using the assumptions that $(s_{a-1,b},$
$s_{a,b-1})\not=(0,0)$, 
$(s_{a-1,0},s_{a,1})\not=(0,0)$, 
$(s_{1,b},s_{0,b-1})\not=(0,0)$, and $(s_{1,0},s_{0,1})$
$\not=(0,0)$, we obtain the following equations:
\begin{equation}\label{eq4}
\begin{aligned}
(e_n^{a-1}e_m^b-t)(e_n^ae_m^{b-1}-t)&=0,\ \ \ \ \ \ \ \ &(e_n^{a-1}-t)(e_n^ae_m-t)&=0,\\
(e_ne_m^b-t)(e_m^{b-1}-t)&=0,\ \ \ \ \ \ \ \ &(e_n-t)(e_m-t)&=0.
\end{aligned}
\end{equation}	
By the equation in $(\ref{eq4})$, either $e_n=t$ or $e_m=t$.

\textbf{Case 1:} We assume that $e_n=e_m=t$.
The equations in $(\ref{eq4})$ are equal to the following equations:
\[	e_n^{a+b-2}=1,\ \ \ \ \ \ (e_n^{a-2}-1)(e_n^a-1)=0,\ \ \ \ \ \ (e_n^b-1)(e_n^{b-2}-1)=0.\]
By \ta second equation, either $e_n^{a-2}=1$ or $e_n^a=1$.
By \ta first equation, if $e_n^{a-2}=1$ (resp. $e_n^a=1$), then $e_n^b=1$ (resp. $e_n^{b-2}=1$).
Thus, ${\rm ord}(f)$ divides ${\rm gcd}(a-2,b)$ (resp. ${\rm gcd}(a,b-2)$).

We assume that $e_n\not=e_m$ and $e_n=t$.
The equations in $(\ref{eq4})$ are equal to the following equations:
\[(e_n^{a-2}e_m^b-1)(e_n^{a-1}e_m^{b-1}-1)=0,\ (e_n^{a-2}-1)(e_n^{a-1}e_m-1)=0,\ (e_m^b-1)(e_m^{b-1}-e_n)=0.\]	

\textbf{Case 2:} We assume that $e_n\not=e_m$, $e_n=t$, $e_n^{a-2}e_m^b=1$, and $e_n^{a-1}e_m^{b-1}=1$.
Since $e_n^{a-2}e_m^b=e_n^{a-1}e_m^{b-1}$, we get $e_n=e_m$.
		This contradicts that $e_n\not=e_m$.
As a result, only one of $e_n^{a-2}e_m^b=1$ and $e_n^{a-1}e_m^{b-1}=1$ holds.

\textbf{Case 3:} We assume that $e_n\not=e_m$, $e_n=t$, 
$e_n^{a-2}e_m^b=1$, $e_n^{a-1}e_m^{b-1}\not=1$, 
and $e_n^{a-2}=e_n^{a-1}e_m=1$.
Since $e_n^{a-2}=e_n^{a-1}e_m$, we get $e_n=e_m^{-1}$.
By $e_n^{a-2}=e_n^{a-2}e_m^b=1$, we have $e_m^b=1$.
Then ${\rm ord}(f)$ divides ${\gcd}(a-2,b)$.

\textbf{Case 4:} We assume that $e_n\not=e_m$, $e_n=t$, 
$e_n^{a-2}e_m^b=1$, $e_n^{a-1}e_m^{b-1}\not=1$, 
$e_n^{a-2}=1$, and $e_n^{a-1}e_m\not=1$.
If $e_m^b=1$ and $e_m^{b-1}=e_n$, then we have $e_ne_m=1$.
By $e_n^{a-2}=1$, we obtain $e_n^{a-1}e_m=1$.
This contradicts that $e_n^{a-1}e_m\not=1$.
Thus, only one of $e_m^b=1$ and $e_m^{b-1}=e_n$ holds.
Since $e_n^{a-2}e_m^b=1$ and $e_n^{a-2}=1$, we have $e_m^b=1$, and hence $e_m^{b-1}\not=e_n$.
Since $e_n=t$, $e_m\not=t$, $e_n^{a-1}e_m^{b-1}\not=1$, $e_n^{a-1}e_m\not=1$, and $e_m^{b-1}\not=e_n$,
we have $s_{0,1}=s_{a,b-1}=s_{a,1}=s_{0,b-1}=0$.
Thus, $F_Q$ is divisible by $X_1$.
Since $X$ is irreducible, $F$ has 
a term of the form $X_1Y_0^jY_1^{b-j}$ for some $2\leq j\leq b-2$.
Since $t=e_n$, we get $e_n=e_m^j$.
Since $e_m^b=1$, ${\rm ord}(f)$ divides $b$.
	
\textbf{Case 5:} We assume that $e_n\not=e_m$, $e_n=t$, 
$e_n^{a-2}e_m^b=1$, $e_n^{a-1}e_m^{b-1}\not=1$, 
$e_n^{a-2}\not=1$, and $e_n^{a-1}e_m=1$.	
By $e_n^{a-2}e_m^b=e_n^{a-1}e_m=1$, we have $e_m^{b-1}=e_n$.
Since $e_n^{a-2}e_m^b=1$, we obtain $e_m^{(a-1)(b-1)+1}=1$.
Thus, ${\rm ord}(f)$ divides $(a-1)(b-1)+1$.
Additionally, if $e_n=e_m^{-1}$,
 then by $e_n^{a-1}e_m=1$, we have $e_n^{a-2}=1$.
This contradicts that $e_n^{a-2}\not=1$.

\textbf{Case 6:} We assume that $e_n\not=e_m$, $e_n=t$, 
$e_n^{a-2}e_m^b\not=1$, $e_n^{a-1}e_m^{b-1}=1$, 
and $e_n^{a-2}=e_n^{a-1}e_m=1$.
Since $e_n^{a-1}e_m^{b-1}=e_n^{a-1}e_m$, we get $e_m^{b-2}=1$.
Since $e_n^{a-2}=e_n^{a-1}e_m$, we have $e_n=e_m^{-1}$.
By $e_n^{a-2}=e_m^b=1$, ${\rm ord}(f)$ divides ${\gcd}(a-2,b)$.

\textbf{Case 7:} We assume that $e_n\not=e_m$, $e_n=t$, 
$e_n^{a-2}e_m^b\not=1$, $e_n^{a-1}e_m^{b-1}=1$, 
and $e_n^{a-2}=1$ and $e_n^{a-1}e_m\not=1$.
Since $e_n^{a-1}e_m^{b-1}=e_n^{a-2}$, we get $e_n=e_m^{1-b}$.
By the third equation in $(\ref{eq4})$, either $e_m^b=1$ or $e_m^{b-1}=e_n$.

If $e_m^b=1$, then by $e_n^{a-2}=1$ and $e_n=e_m^{1-b}$, we obtain that $e_n=e_m^{-1}$ and ${\rm ord}(f)$ divides ${\rm gcd}(a-2,b)$. 

If $e_m^{b-1}=e_n$, then $e_n=e_m^{1-b}$, we have $e_n=-1$.
Since $e_n^{a-2}=e_n^{a-1}e_m$, we have $e_n=e_m^{-1}$.
By $e_n^{a-2}=e_m^b=1$, ${\rm ord}(f)$ divides ${\gcd}(a-2,b)$.
Additionally, if $e_n=e_m^{-1}$, then $e_m=-1$.
This contradicts that $e_n\not=e_m$. 

\textbf{Case 8:} We assume that $e_n\not=e_m$, $e_n=t$, 
$e_n^{a-2}e_m^b\not=1$, $e_n^{a-1}e_m^{b-1}=1$, 
and $e_n^{a-2}\not=1$ and $e_n^{a-1}e_m=1$.
Since $e_n^{a-1}e_m^{b-1}=e_n^{a-1}e_m$, we get $e_m^{b-2}=1$.
By the third equation in $(\ref{eq4})$, either $e_m^b=1$ or $e_m^{b-1}=e_n$.

If $e_m^b=1$, then by $e_m^{b-2}=1$, we get $e_m=-1$.
Since $e_n^{a-2}=e_n^{a-1}e_m$, we have $e_n=e_m^{-1}$.
By $e_n^{a-1}e_m^{b-1}=e_m^b=1$, we obtain $e_n^{a-1}=-1$.
Then $(e_n^2)^{a-1}=1$.
If $f^2=[D(e_n^2,1)]\times[I_2]\not={\rm id}_{C_{a,b}}$, then ${\rm ord}(f^2)$ divides $a-1$.
By Lemma \ref{2}, this is a contradiction.
Then $e_n^2=1$, i.e. $e_n=-1$.
This contradicts that $e_n\not=e_m$.

If $e_m^{b-1}=e_n$, then by $e_m^{b-2}=1$, we obtain that $e_n=e_m$.
This contradicts that $e_n\not=e_m$.

We assume that $e_n\not=e_m$ and $e_m=t$.
By the symmetry of $a,n$ and $b,m$ in the equations $(\ref{eq4})$,
we get that this proposition.
\end{proof}
\begin{thm}\label{16}
Let $C_{a,b}\subset \pp$ be a smooth curve of bidegree $(a,b)$ with an automorphism $f$.
We assume that $|\{Q_i\}_{i=1}^4 \cap C_{a,b}|=4$.
\begin{enumerate}
\item[$(i)$]If $f=[D(e_n,1)]\times [D(e_m,1)]$ where $n,m\geq 2$,
then ${\rm ord}(f)$ divides either $a$, $b$, or $(a-1)(b-1)+1$.
\item[$(ii)$]If $f=s_{(1,2)}\circ ([A]\times [B])$ where $A$ and $B$ are diagonal matrices such that $AB=BA=D(e_n,1)$ with $n\geq2$,
		then ${\rm ord}(f)=4$.
\item[$(iii)$]If $a\geq 3$, then ${\rm ord}(f)$ does not divide $kl$ where $k\geq2$ and $l:={\rm Max}\{a,b\}$.
	\end{enumerate}	
\end{thm}
\begin{proof}
First, we assume that $f=[D(e_n,1)]\times [D(e_m,1)]$ where $n,m\geq 2$.
Since ${\rm gcd}(a-2,b)$ divides $b$ and ${\rm gcd}(a,b-2)$ divides $a$, by Proposition \ref{16-}, ${\rm ord}(f)$ divides either
$a$, $b$, or $(a-1)(b-1)+1$.
	
Next, we assume that $f=s_{(1,2)}\circ ([A]\times [B])$ where $A$ and $B$ are diagonal matrices such that $AB=BA=D(e_n,1)$ with $n\geq2$.
By Lemma \ref{8}, $a=b$.
Since $f^2=[D(e_n,1)]\times[D(e_n,1)]$,
$f^2$ corresponds to part $(i)$ of Proposition \ref{16-}.
Since $a=b$, ${\rm gcd}(a-2,b)={\rm gcd}(a,b-2)={\rm gcd}(a,2)$. 
Then ${\rm ord}(f^2)=2$, and hence ${\rm ord}(f)=4$.

From the above, parts $(i)$ and $(ii)$ of this theorem are shown.

Finally, we assume that $a\geq 3$ and ${\rm ord}(f)$ divides $lm$ where $l\geq2$ and $m:={\rm Max}\{a,b\}$.
By parts $(i)$ and $(ii)$ of this theorem,
 ${\rm ord}(f)$ divides $(a-1)(b-1)+1$.
Without loss of generality, 
we may assume that $a={\rm Max}\{a,b\}$.
Then $ka$ divides $(a-1)(b-1)+1=a(b-1)-(b-2)$.
As a result, $a$ divides $b-2$.
This contradicts that $a={\rm Max}\{a,b\}$.
\end{proof}

\begin{thm}\label{pmain}
Let $C_{a,b}\subset \pp$ be a smooth curve of bidegree $(a,b)$, and let $f$ be an automorphism of $C_{a,b}$ where $a,b\geq 3$.
Then we have the following:
\begin{enumerate}
\item[$(i)$]${\rm ord}(f)$ divides either $6$, $k-2$, $2(k-1)$, $(a-1)(b-1)+1$, $a(b-1)$, $(a-1)b$, 
or $ab$ where $k\in\{a,b\}$.
\item[$(ii)$]If $|{\rm Fix}(f)|>0$ and ${\rm ord}(f)$ divides $kl$ where $k\geq2$ and $l:={\rm Max}\{a,b\}$,
then $C_{a,b}/\langle f\rangle \cong \mathbb P^1$.
\end{enumerate}
\end{thm}
\begin{proof}
We set $m:={\rm Max}\{a,b\}$.
By Lemmas \ref{-1} and \ref{8}, 
we may assume that $f=[D(e_n,1)]\times [D(e_m,1)]$ or $f=\s\circ ([A]\times [B])$ where $A,B\in{\rm GL}(\mathbb C,2)$ are diagonal matrices such that $AB=BA=D(e_n,1)$, and $n$ and $m$ are positive integers.

First, we assume that $f=[D(e_n,1)]\times [D(e_m,1)]$ and either $n=1$ or $m=1$. 
By Lemma \ref{2}, ${\rm ord}(f)$ divides $a$ or $b$.
Thus, ${\rm ord}(f)$ divides $ab$.

Next we assume that $f=[D(e_n,1)]\times [D(e_m,1)]$ or $f=\s\circ ([A]\times [B])$ where $A,B\in{\rm GL}(\mathbb C,2)$ are diagonal matrices such that $AB=BA=D(e_n,1)$ and $n,m\geq2$.
By Lemma \ref{3}, $|C_{a,b}\cap \{Q_i\}_{i=1}^4|\not=1$.

We assume that $|C_{a,b}\cap \{Q_i\}_{i=1}^4|=0$.
Since $2k$ divides either $(a-1)b$, $a(b-1)$, or $ab$ where $k\in\{a,b\}$, 
by Theorem \ref{4}, ${\rm ord}(f)$ divides either $(a-1)b$, $a(b-1)$, or $ab$. 
Furthermore, if ${\rm ord}(f)$ divides $lm$ where $l\geq2$, then ${\rm Fix}(f)=\emptyset$.

We assume that $|C_{a,b}\cap \{Q_i\}_{i=1}^4|=2$.
Since $4$ and $2k$ divide either $2(k-1)$, $(a-1)b$, $a(b-1)$, or $ab$ where $k\in\{a,b\}$,
by Theorem \ref{11}, ${\rm ord}(f)$ divides either $2(k-1)$, $(a-1)b$, $a(b-1)$, or $ab$ where $k\in\{a,b\}$.
Furthermore, if ${\rm ord}(f)$ divides $lm$ where $l\geq2$, then ${\rm Fix}(f)=\emptyset$.

We assume that $|C_{a,b}\cap \{Q_i\}_{i=1}^4|=3$.
By Theorem \ref{13}, ${\rm ord}(f)$ divides either $6$ or $k-2$ where $k\in\{a,b\}$, and
${\rm ord}(f)$ does not divide $lm$ where $l\geq2$.

We assume that $|C_{a,b}\cap \{Q_i\}_{i=1}^4|=4$.
Since $4$ and $2k$ divide either $2(k-1)$ or $ab$ where $k\in\{a,b\}$,
by Theorem \ref{13}, ${\rm ord}(f)$ divides either $2(k-1)$, $ab$, or $(a-1)(b-1)+1$ where $k\in\{a,b\}$, and
${\rm ord}(f)$ does not divide $lm$ where $l\geq2$.
\end{proof}
Note that when $a\geq 3$, among $(a-1)^2+1$, $a(a-1)$, and $a^2$, the smallest is $(a-1)^2+1$. 
\begin{lem}\label{lpmain2}
Let $C_{a,b}\subset \pp$ be a smooth curve of bidegree $(a,b)$ with an automorphism $f$ where $a,b\geq 4$.
If $f=s_{(1,2)}\circ ([A]\times [B])$ where $A$ and $B$ are diagonal matrices such that $AB=BA=D(e_n,1)$ with $n\geq2$,
then $a=b$, and 
${\rm ord}(f)$ is none of $(a-1)^2+1$, $a(a-1)$, or $a^2$.
\end{lem}
\begin{proof}
By Lemma \ref{8}, $a=b$.
Since $f^2=[D(e_n,1)]\times[D(e_n,1)]$ where $n\geq2$,
by Lemma \ref{3}, $|C_{a,b}\cap \{Q_i\}_{i=1}^4|\not=1$.
By Theorems \ref{4}, \ref{11},\ref{13}, and \ref{16},
${\rm ord}(f)$ divides either $2a$, $4$, $2(a-1)$, or $6$.
Since $a\geq 4$, among $2a$, $4$, $2(a-1)$, and $6$, the largest is $2a$.  
Since $a\geq4$, $2a<(a-1)^2+1$. Thus, ${\rm ord}(f)$ is none of $(a-1)^2+1$, $a(a-1)$, or $a^2$.
\end{proof}
\begin{thm}\label{pmain2}
Let $a,b\geq 4$ be integers.
Let $C_{a,b}\subset \pp$ be a smooth curve of bidegree $(a,b)$ with an automorphism $f$, and $\ff$ be the defining equation of $C_{a,b}$.
\begin{enumerate}
\item[$(i)$]If ${\rm ord}(f)=ab$, then $C_{a,b}/\langle f\rangle \cong \mathbb P^1$ and 
by replacing the coordinate system if necessary,
\[f=[D(e_a,1)]\times[D(e_b,1)]\] and
 \[\ff=X_0^aY_0^b+X_0^aY_1^b+X_1^aY_0^b+sX_1^aY_1^b\] for some $s\in\mathbb C^*$.
\item[$(ii)$]If ${\rm ord}(f)=(a-1)b$, then $C_{a,b}/\langle f\rangle \cong \mathbb P^1$ and
by replacing the coordinate system if necessary,
\[f=[D(e_{(a-1)b}^{-b},1)]\times[D(e_{(a-1)b},1)]\] and
\[\ff=X_0^aY_0^b+X_0^{a-1}X_1Y_1^b+X_0X_1^{a-1}Y_0^b+sX_1^aY_1^b\] for some $s\in\mathbb C^*$.
\item[$(iii)$]If ${\rm ord}(f)=a(b-1)$, then $C_{a,b}/\langle f\rangle \cong \mathbb P^1$ and
by replacing the coordinate system if necessary, 
\[f=[D(e_{a(b-1)},1)]\times[D(e_{a(b-1)}^{-a},1)]\] and \[\ff=X_0^aY_0^b+X_1^aY_0^{b-1}Y_1+X_0^aY_0Y_1^{b-1}+sX_1^aY_1^b\] for some $s\in\mathbb C^*$.
\item[$(iv)$]If ${\rm ord}(f)=(a-1)(b-1)+1$, then $C_{a,b}/\langle f\rangle \cong \mathbb P^1$ and
by replacing the coordinate system if necessary, 
either 
\[f=[D(e_{(a-1)(b-1)+1}^{b-1},1)]\times[D(e_{(a-1)(b-1)+1},1)]\] and 
\begin{equation*}
\begin{split}
\ff=&X_0^{a-1}X_1Y_0^b+X_0^aY_0Y_1^{b-1}\\
&+sX_1^aY_0^{b-1}Y_1+s'X_0X_1^{a-1}Y_1^b
\end{split}
\end{equation*}
 or 
\[f=[D(e_{(a-1)(b-1)+1},1)]\times[D(e_{(a-1)(b-1)+1}^{a-1},1)]\] and
\begin{equation*}
\begin{split}
\ff=&X_0^aY_0^{b-1}Y_1+X_0X_1^{a-1}Y_0^b\\
&+sX_0^{a-1}X_1Y_1^b+s'X_1^aY_0Y_1^{b-1}
\end{split}
\end{equation*}
 for some $(s,s')\in(\mathbb C^*)^2$.
\end{enumerate}
Here, $s,t,u,v\in\mathbb C^{\ast}$.
\end{thm}
\begin{proof}
By Lemmas \ref{-1}, \ref{8}, \ref{2}, and \ref{pmain2},
we may assume that $f=[D(e_n,1)]\times [D(e_m,1)]$ where $n,m\geq2$.

We assume that ${\rm ord}(f)=ab$.
Since $a,b\geq 4$, and by Lemma \ref{3}, and Theorems \ref{4}, \ref{11}, \ref{13}, and \ref{16},
we get that $|C_{a,b}\cap \{Q_i\}_{i=1}^4|=0$.
By Proposition \ref{4-}, and $a$ and $b$ are coprime, we may assume that $f=[D(e_a,1)]\times [D(e_b,1)]$ and
\begin{equation*}
\begin{split}
\ff=&s_{a,b}X_0^aY_0^b+s_{a,0}X_0^aY_1^b+s_{0,b}X_1^aY_0^b+s_{0,0}X_1^aY_1^b\\
&+\sum_{(i,j)\in\mI}s_{i,j}X_0^iX_1^{a-i}Y_0^jY_1^{b-j}
\end{split}
\end{equation*}
where $s_{a,b},s_{a,0},s_{0,b},s_{0,0}\in\mathbb C^{\ast}$ and $s_{i,j}\in \mathbb C$ for $(i,j)\in\mI$.
Since $f$ is an automorphism of $C_{a,b}$, we have
$f^{\ast}\ff=t\ff$ for some $t\in\mathbb C^{\ast}$.
Since $f=[D(e_a,1)]\times [D(e_b,1)]$, and $\ff$ has a term of the form $X_1^aY_1^b$, we get that $t=1$.
If $s_{i,j}\not=0$ for $(i,j)\in\mI$,
then we have $e_a^ie_b^j=1$, i.e. $e_a^i=e_b^{-j}$.
Since $a$ and $b$ are coprime, $0\leq i\leq a$, and $0\leq j\leq b$,
we get that
$(i,j)$ is either $(0,0)$, $(a,0)$, $(0,b)$, or $(a,b)$.
Then $(i,j)\in\mathbb E$.
This contradicts that $(i,j)\in\mI$.
As a result, $s_{i,j}=0$ for $(i,j)\in\mI$.
Then $\ff=s_{a,b}X_0^aY_0^b+s_{a,0}X_0^aY_1^b+s_{0,b}X_1^aY_0^b+s_{0,0}X_1^aY_1^b$.
Since multiplying $\ff$ by $s_{a,b}^{-1} \in \mathbb{C}^*$ does not affect the defining equation of $C_{a,b}$, we may assume $s_{a,b} = 1$. Moreover, by applying the coordinate changes $Y_1 \mapsto \left(\sqrt[b]{s_{a,0}^{-1}}\right)Y_1$ and $X_1 \mapsto \left(\sqrt[a]{s_{0,b}^{-1}}\right)X_1$, the automorphism $f = [D(e_a,1)] \times [D(e_b,1)]$ remains unchanged, and the defining equation of $C_{a,b}$ becomes $X_0^aY_0^b + X_0^aY_1^b + X_1^aY_0^b + s_{0,0}X_1^aY_1^b$
where $s_{0,0}\in\mathbb C^*$.
In addition, since $a$ and $b$ are coprime, $e_a^b$ is a primitive $a$-th root of unity.
By $f^b=[D(e_a^b,1)]\times [I_2]$ and Lemma \ref{th}, we get that $C_{a,b}/\langle f \rangle\cong \mathbb P^1$.

We assume that ${\rm ord}(f)=(a-1)b$.
Since $a,b\geq 4$, and by Lemma \ref{3}, and Theorems \ref{4}, \ref{11},\ref{13}, and \ref{16},
we get that $|C_{a,b}\cap \{Q_i\}_{i=1}^4|=2$.
By Corollary \ref{-2} and Proposition \ref{10},
we get that
$C_{a,b}\cap \{Q_i\}_{i=1}^4=\{Q_2,Q_3\}$, $f=[D(e_{(a-1)b}^{-b},1)]\times [D(e_{(a-1)b},1)]$, and
\begin{equation*}
\begin{split}
\ff=&s_{a,b}X_0^aY_0^b+s_{a-1,0}X_0^{a-1}X_1Y_1^b+s_{1,b}X_0X_1^{a-1}Y_0^b+s_{0,0}X_1^aY_1^b\\
&+\sum_{(i,j)\in\mI}s_{i,j}X_0^iX_1^{a-i}Y_0^jY_1^{b-j}
	\end{split}
\end{equation*}
where $s_{a,b},s_{a-1,0},s_{1,b},s_{0,0}\in\mathbb C^{\ast}$ and $s_{i,j}\in \mathbb C$ for $(i,j)\in\mI$.
Since $f$ is an automorphism of $C_{a,b}$, we get
$f^{\ast}\ff=t\ff$ for some $t\in\mathbb C^{\ast}$.
Since $f=[D(e_{(a-1)b}^{-b},1)]\times [D(e_{(a-1)b},1)]$, and $\ff$ has a term of the form $X_1^aY_1^b$, we obtain $t=1$.
If $s_{i,j}\not=0$ for $(i,j)\in\mI$,
then we have $e_{(a-1)b}^{-ib}e_{(a-1)b}^j=1$, i.e. $e_{(a-1)b}^{ib}=e_{(a-1)b}^j$.
Since $0\leq i\leq a$ and $0\leq j\leq b$,
we get that $(i,j)$ is either $(0,0)$, $(a-1,0)$, $(0,b)$, or $(a,b)$.
Then $(i,j)\in\mathbb E$.
This contradicts that $(i,j)\in\mI$.
As a result, $s_{i,j}=0$ for $(i,j)\in\mI$.
Then $\ff=s_{a,b}X_0^aY_0^b+s_{a-1,0}X_0^{a-1}X_1Y_1^b+s_{1,b}X_0X_1^{a-1}Y_0^b+s_{0,0}X_1^aY_1^b$.
Since multiplying $\ff$ by $s_{a,b}^{-1} \in \mathbb{C}^*$ does not affect the defining equation of $C_{a,b}$, we may assume $s_{a,b} = 1$. Moreover, by applying the coordinate changes $Y_1 \mapsto \left(\sqrt[b]{s_{a-1,0}^{-1}}\right)Y_1$ and $X_1 \mapsto \left(\sqrt[a-1]{s_{1,b}^{-1}}\right)X_1$, the automorphism $f = [D(e_{(a-1)b}^{-b},1)]\times [D(e_{(a-1)b},1)]$ remains unchanged, and the defining equation of $C_{a,b}$ becomes $X_0^aY_0^b+X_0^{a-1}X_1Y_1^b+X_0X_1^{a-1}Y_0^b+s_{0,0}X_1^aY_1^b$
where $s_{0,0}\in\mathbb C^*$.
In addition, since $f^{a-1}=[I_2]\times [e_{(a-1)b}^{a-1}]$,
$e_{(a-1)b}^{a-1}$ is a primitive $b$-th root of unity, and Lemma \ref{th}, we get that $C_{a,b}/\langle f \rangle\cong \mathbb P^1$.
The case where ${\rm ord}(f)=a(b-1)$ is shown in a similar manner.

We assume that ${\rm ord}(f)=(a-1)(b-1)+1$.
Since $a,b\geq 4$, and by Lemma \ref{3}, and Theorems \ref{4}, \ref{11},\ref{13}, and \ref{16},
we get that $|C_{a,b}\cap \{Q_i\}_{i=1}^4|=4$.
By Proposition \ref{16-},
$f$ is $[D(e_{(a-1)(b-1)+1}^{b-1},1)]\times [D(e_{(a-1)(b-1)+1},1)]$ or $[D(e_{(a-1)(b-1)+1},1)]\times [D(e_{(a-1)(b-1)+1}^{a-1},1)]$.
We assume that $f=[D(e_{(a-1)(b-1)+1}^{b-1},1)]\times [D(e_{(a-1)(b-1)+1},1)]$.
By part $(iv)$ of Proposition \ref{16-},
\begin{equation*}
	\begin{split}
		\ff=&s_{a-1,b}X_0^{a-1}X_1Y_0^b+s_{a,1}X_0^aY_0Y_1^{b-1}+s_{0,b-1}X_1^aY_0^{b-1}Y_1\\
		&+
		s_{1,0}X_0X_1^{a-1}Y_1^b+\sum_{(i,j)\in\mI}s_{i,j}X_0^iX_1^{a-i}Y_0^jY_1^{b-j}
	\end{split}
\end{equation*}
where $s_{a-1,b},s_{a,0},s_{0,b-1},s_{1,0}\in\mathbb C^{\ast}$ and $s_{i,j}\in \mathbb C$ for $(i,j)\in\mI$.
Since $f$ is an automorphism of $C_{a,b}$, we have
$f^{\ast}\ff=t\ff$ for $t\in\mathbb C^{\ast}$.
Since $f=[D(e_{(a-1)(b-1)+1}^{b-1},1)]\times [D(e_{(a-1)(b-1)+1},1)]$, and $\ff$ has a term of the form $X_0X_1^{a-1}Y_1^b$, we get that $t=e_{(a-1)(b-1)+1}^{b-1}$.
If $s_{i,j}\not=0$ for $(i,j)\in\mI$,
then $e_{(a-1)b}^{i(b-1)}e_{(a-1)b}^j=e_{(a-1)(b-1)+1}^{b-1}$, i.e. $e_{(a-1)(b-1)+1}^{(i-1)(b-1)+j}=1$.
Since $0\leq i\leq a$ and $0\leq j\leq b$,
we get that $(i,j)$ is either $(0,b-1)$, $(1,0)$, $(a-1,b)$, or $(a,1)$.
Then $(i,j)\in\mathbb E$.
This contradicts that $(i,j)\in\mI$.
As a result, $s_{i,j}=0$ for $(i,j)\in\mI$.
Then $\ff=s_{a-1,b}X_0^{a-1}X_1Y_0^b+s_{a,1}X_0^aY_0Y_1^{b-1}+s_{0,b-1}X_1^aY_0^{b-1}Y_1+s_{1,0}X_0X_1^{a-1}Y_1^b$.
Since multiplying $\ff$ by $s_{a-1,b}^{-1} \in \mathbb{C}^*$ does not affect the defining equation of $C_{a,b}$, we may assume $s_{a-1,b} = 1$. Moreover, by applying the coordinate changes $Y_1 \mapsto \left(\sqrt[b-1]{s_{a,1}^{-1}}\right)Y_1$, the automorphism $f =[D(e_{(a-1)(b-1)+1}^{b-1},1)]\times [D(e_{(a-1)(b-1)+1},1)]$ remains unchanged, and the defining equation of $C_{a,b}$ becomes $X_0^{a-1}X_1Y_0^b+X_0^aY_0Y_1^{b-1}+s_{0,b-1}X_1^aY_0^{b-1}Y_1+s_{1,0}X_0X_1^{a-1}Y_1^b$
where $s_{0,b-1},s_{1,0}\in\mathbb C^*$.

Let $p\co C_{a,b}\ra C_{a,b}/\langle f\rangle $ be the quotient morphism.
We set $G:=\langle f\rangle $
For $x\in G$, let $G_x:=\{g\in G\,|\,g(x)=x\}$.
Then the ramification index at $x$ is equal to $|G_x|$.
Let $g$ and $g'$ denote the genus of $C_{a,b}$ and $C_{a,b}/G$, respectively.
By the Riemann-Hurwitz formula,
\[2-2g+\sum_{x\in C_{a,b}}(|G_x|-1)=|G|(2-g').\]
Since the bidegree of $C_{a,b}$ is $(a,b)$, $g=(a-1)(b-1)$.
Since the genus is a non-negative integer, $C_{a,b}/G$ being $\mathbb P^1$ is equivalent to $2-2g'>0$.
Therefore, to show that 
$C_{a,b}/G\cong \mathbb P^1$, it suffices to show that 
\[\sum_{x\in C_{a,b}}(|G_x|-1)>2(ab-a-b).\]
Since $Q_i\in {\rm Fix}(f)$, $G_{Q_i}=G$ for $i=1,\ldots,4$.
Since $|G|=(a-1)(b-1)+1=ab-a-b+2$
\[\sum_{x\in C_{a,b}}(|G_x|-1)\geq 4(ab-a-b+2)>2(ab-a-b).\]
As a result, $C_{a,b}/G\cong \mathbb P^1$.

The case where $f=[D(e_{(a-1)(b-1)+1},1)]\times [D(e_{(a-1)(b-1)+1}^{a-1},1)]$ is shown in a similar manner.
\end{proof}

	%
%
	

\begin{thebibliography}{99}
		\bibitem{ha}\label{bio:acgh}
		E. Arbarello, M. Cornalba, P.A. Griffiths, J. Harris, Geometry of algebraic curves, vol. I. Grundlehren der Mathematischen Wissenschaften 267. Springer, New York (1985).
\bibitem{ha}\label{bio:bb2016}
E. Badr and F. Bars, Non-singular plane curves with an element of “large” order in its automorphism group, Int. J. Algebra Comput. 26 (2016), 399$-$434.
\bibitem{ha}\label{bio:bb}
E. Badr and F. Bars, Automorphism groups of non-singular plane curves of degree 5 , Commun. Algebra 44 (2016), 327$-$4340.
\bibitem{ha}\label{bio:bb25}
E. Badr, F. Bars, The stratification by automorphism groups of smooth plane sextic curves. Annali di Matematica (2025). https://doi.org/10.1007/s10231-025-01558-z.
\bibitem{ha}\label{bio:harts}
R. Hartshorne, Algebraic geometry, Graduate Texts in Mathematics, No. 52. Springer-Verlag, New York, Heidelberg, 1977.
\bibitem{ha}\label{bio:haru}		 
T. Harui, Automorphism groups of smooth plane curves. Kodai Math. J. 42(2), 308$-$331 (2019)
\bibitem{ha}\label{bio:th21l}
T. Hayashi, Linear automorphisms of smooth hypersurfaces giving Galois points, Bull. Korean Math. Soc. 58 (2021), No. 3, pp. 617$-$635.
\bibitem{ha}\label{bio:th21}
T. Hayashi, Orders of automorphisms of smooth plane curves for the automorphism groups to be cyclic, Arab. J. Math. 10, 409-422 (2021).
\bibitem{ha}\label{bio:th23g}
T. Hayashi, Galois covers of the projective line by smooth plane curves of large degree, Beitr Algebra Geom 64, 311-365 (2023). 

\bibitem{ha}\label{bio:hen}
P. Henn, Die Automorphismengruppen dar algebraischen Functionenkorper vom Geschlecht 3. Inagural-dissertation, Heidelberg (1976).
\bibitem{ha}\label{bio:kuko}
A. Kuribayashi and K. Komiya, OnWeierstrass points of non-hyperelliptic compact Riemann surfaces of genus three. Hiroshima
Math. J. 7, 743$-$786 (1977).
\bibitem{ha}\label{bio:my00}
K. Miura and H. Yoshihara, Field theory for function fields of plane quartic curves, J. Algebra, 226 (2000), 283$-$294.
\bibitem{ha}\label{bio:tt12}
T. Takahashi. Galois morphism computing the gonality of a curve on a Hirzebruch surface. J. Pure Appl. Algebra, 216 (2012), 12$-$19.
\bibitem{ha}\label{bio:y01f}
H. Yoshihara, Function field theory of plane curves by dual curves, J. Algebra 239 (2001), 340$-$355.
\end{thebibliography}
\end{document}